\documentclass{amsart}
\usepackage{nicefrac}

  \newcommand{\ie}{\textit{i.e.}}%
  \newcommand{\nff}{\emph{n.f.f.}}
  \newcommand{\nc}{noncommutative}

  \newcommand{\MM}{\ensuremath{\mathcal{M}}}%
  \newcommand{\NN}{\ensuremath{\mathcal{N}}}%
  \newcommand{\dist}{\ensuremath{{\rm dist}}}%
  \newcommand{\C}{\ensuremath{\mathbb{C}}}%
  \newcommand{\R}{\ensuremath{\mathbb{R}}}%
  \newcommand{\N}{\ensuremath{\mathbb{N}}}%
  \newcommand\M[2]{\ensuremath{M_{#1}\left(#2\right)}}
  \newcommand\Mn{\ensuremath{M_n}}
  \newcommand\Mm{\ensuremath{M_m}}
  \newcommand{\paren}[1]{\left( #1 \right)}
  \newcommand\id{\ensuremath{\mathrm{id}}}
  \newcommand{\egdef}{\ensuremath{\stackrel{\text{\tiny def}}=}}%

  \newcommand\B[1]{\ensuremath{B\left(#1\right)}}
  \newcommand\tr{\ensuremath{\mathrm{tr}}}%
  \newcommand\Tr{\ensuremath{\mathrm{Tr}}}%
  \newcommand{\un}{\ensuremath{{1}}}%

\newtheorem{thm}{Theorem}[section]
\newtheorem{lemma}[thm]{Lemma}
\newtheorem{corollaire}[thm]{Corollary}

\newtheorem{proposition}[thm]{Proposition}

\theoremstyle{remark}
\newtheorem{rem}[thm]{Remark}

\theoremstyle{definition}
\newtheorem{dfn}{Definition}[section]
\newtheorem{dfn-prop}{Definition-Proposition}

\begin{document}
\author{Mikael de la Salle}
\title{Complete isometries between subspaces of noncommutative $L_p$-spaces}
\address{ {\'E}quipe d'Analyse Fonctionnelle \\ Institut de
  Math{\'e}matiques de Jussieu \\ Universit{\'e} Paris 6}
\address{D{\'e}partement de math{\'e}matiques et applications \\ {\'E}cole
  Normale Sup{\'e}rieure\\45 rue d'Ulm\\ 75230 Paris Cedex 05\\ France}
\email{mikael.de.la.salle@ens.fr} 

%j'ai enlevé 46B04 = Isometric theory of Banach spaces
% 46L51 Noncommutative measure and integration
% 46L52 Noncommutative function spaces
% 46L07 Operator spaces and completely bounded maps 

\begin{abstract}
  We prove some noncommutative analogues of a theorem by Plotkin and Rudin
  about isometries between subspaces of $L_p$-spaces.

  Let $0<p<\infty$, $p$ not an even integer. The main result of this paper
  states that in the category of unital subspaces of noncommutative
  probability $L_p$-spaces, under some boundedness condition, the unital
  completely isometric maps come from $*$-isomorphisms of the underlying
  von Neumann algebras.

  Some applications are given, including to non commutative $H^p$
  spaces.
\end{abstract}
\maketitle

\section*{Introduction}
The study of isometries between Banach spaces has been an active area
of research in the theory of Banach spaces for a long time, see for
example the survey \cite{MR1957004}. The isometries between $L_p$
spaces were first described by Banach, with a final proof given by
Lamperti. The study of isometries between subspaces of $L_p$-spaces,
goes back at least to the $1960$'s with Forelli's work
\cite{MR0169081}, but the most general result is due independently to
Plotkin in a series of articles in the $1970'$s \cite{MR0621703} and
to Rudin in \cite{MR0410355}; see also Hardin \cite{MR611233}. The
reader is refered the \cite[Chapter 2]{MR1863709} for a survey. 

The study of isometries between whole \nc{} $L_p$ spaces has
already interested a few mathematicians, and the final characterization
(in the tracial case) was given by Yeadon in \cite{MR611284}. Recent
results were also obtained for non semifinite von Neumann algebras
(\cite{MR2215435} and \cite{MR2255812}). The study of complete
isometries between \nc{} $L_p$ has also been more recently
studied by Junge, Ruan and Sherman in \cite{MR2255812}.

In this paper we will be interested in the study of complete
isometries between subspaces of \nc{} $L_p$ spaces, and the
main results are close analogues of the result for isometries in
subspaces of classical $L_p$ spaces.

We first recall Plotkin's and Rudin's theorem:
\begin{thm}[Plotkin, Rudin]
\label{thm=theoreme_1_rudin}
Let $0<p<\infty$ and $p \neq 2,4,6,8,\dots$. Let $\mu$ and $\nu$ be two
probability measures (on arbitrary measure spaces $\Omega$ and
$\Omega'$). Let finally $n$ be a positive integer and $f_1, \dots f_n
\in L_p(\mu)$, $g_1, \dots g_n \in L_p(\nu)$.

Assume that for all complex numbers $z_1, \dots z_n \in \C$, 
\begin{equation}
\label{eq=egalite_des_normes_p_cas_commutatif}
\int \left |1 + z_1 f_1 + \dots z_n f_n\right|^p d \mu = 
\int \left |1 + z_1 g_1 + \dots z_n g_n\right|^p d \nu.
\end{equation}

Then $(f_1 ,\dots f_n)$ and $(g_1 ,\dots g_n)$ form two equimeasurable
families. Probabilistically, this means that the $\R^n$-valued random
variables $(f_1 ,\dots f_n)$ and $(g_1 ,\dots g_n)$ have the same distribution.
\end{thm}

The following theorem was also proved by Rudin in his paper
\cite{MR0410355}. It had previously been proved in weaker forms by
Forelli (\cite{MR0169081} and \cite{MR0318897}).
\begin{thm}[Rudin]
\label{thm=theoreme_2_rudin}
Let $\mu$ and $\nu$ be as above, and $0<p< \infty$, $p \neq 2$. Let
$M \subset L_p(\mu)$ be a (complex) unital algebra (with respect to
the point-wise product), and $A:M
\rightarrow L_p(\nu)$ a unital linear isometry: $A(1)=1$ and
\[ \int |f|^p d\mu = \int |A(f)|^p d\nu \ \ \ \ \ \ \ \ \forall f \in M.\]

\begin{itemize}
\item Then for all $f,g \in M$:
\[A(f g) = A(f)A(g) \ \ \ \ \ \ \ \ \forall f,g \in M\]
and
\[\|A(f)\|_\infty = \|f\|_\infty.\]
\item If moreover $M \subset L_\infty$ or $p \neq 4,6,8,\dots$, then
  for all $n$ and $f_1, \dots f_n \in M$, $(f_1 ,\dots f_n)$ and $(A
  f_1 ,\dots A f_n)$ are equimeasurable.
\end{itemize}
\end{thm}

\bigskip

In this paper similar results are proved in the \nc{} setting (with
some additional boundedness conditions). The commutative $L_p$-spaces
have to be replaced by \nc{} spaces $L_p(\MM,\tau)$ associated to a von
Neumann algebra $\MM$ with a finite normalized trace $\tau$, and
isometries are replaced by complete isometries. Let us briefly
introduce the vocabulary.

In the whole paper $(\MM,\tau)$ and $(\NN, \widetilde \tau)$ are von Neumann
algebras equipped with normal faithful finite (\nff{}) traces. The units of
$\MM$ and $\NN$ are denoted by $\un_\MM$ and $\un_\NN$ or simply by $\un$.
The traces will always be assumed to be normalized: $\tau(\un)=1$. 

When $n$ is an integer, the set of $\MM$-valued $n \times n$ matrices
is denoted by $\M{n}{\MM}$, is identified with the tensor product $\Mn
\otimes \MM$ and is provided with a normal faithful tracial state $\tau
\sp {(n)} \egdef \tr_n \otimes \tau$. Here $\tr_n$ 
denotes the normalized trace on $\Mn$:
\[\tr_n(a) = \frac 1 n \Tr(a) =  \frac 1 n \sum_{1 \leq j \leq n}
a_{j,j}.\]

The unit of $\M{n}{\MM}$ is $\un_n \otimes \un_\MM$ and will be denoted
simply by $\un$ when no confusion is possible.

Let $0<p<\infty$. If $x \in \MM$, the ``$p$-norm'' of $x$ is denoted by
$\|x\|_p$ and is equal to
\[\|x\|_p = \|x\|_{L_p(\tau)} \egdef \paren{\tau\paren{|x|^p}}^{\nicefrac 1
  p}.\] In the same way, if $x \in \M n \MM$, $\|x\|_p$ denotes the quantity
$\|x\|_{L_p(\tau \sp{(n)})}$. Remark that $\| \cdot \|_p$ is a norm only if
$p \geq 1$.  In this case, $L_p(\MM,\tau)$ is defined as the completion of
$\MM$ with respect to the norm $\| \cdot \|_p$ (see the survey
\cite{MR1999201} for more details, see also section
\ref{part=non_borne}). If $p=\infty$, $L_\infty(\MM,\tau)$ is just $\MM$ with
the operator norm. The space $L_p(\MM,\tau)$ will be denoted by $L_p(\MM)$ of
$L_p(\tau)$ when no confusion is possible.

As usual, the main modification one has to bring in order to deal with
the non commutativity is the fact that one has to allow operator
coefficients instead of scalar coefficients in
\eqref{eq=egalite_des_normes_p_cas_commutatif}.

In the whole paper, we will try to use the following notation: unless
explicitly specified, small letters $x$ or $y$ will stand for elements of
the von Neumann algebras $\MM$ or $\NN$, $a$, $b$ will stand for finite
complex-valued matrices viewed as matricial coefficients. Operators written
with capital letters will be matrices with coefficients in $\MM$ or
$\NN$. The letters $z$ and $\lambda$ (resp. $s$ and $t$) will denote
complex (resp. real) numbers. In a typical equation like
\[ S = \sum_k z_k a_k \otimes x_k \in \M n \MM,\] it should thus be clear to
which set all the $z_k$, $a_k$ and $x_k$ belong.

\bigskip

At least a far as bounded operators are concerned, the fact that two
families of \nc{} random variables (\ie{}  elements of the
$L_p$-spaces) are equimeasurable can be expressed by requesting that they
have the same $*$-distributions. Let us recall the definition of the
distribution of \nc{} random variables. If $(x_i)_{i \in I} \subset \MM$ is
a family of operators in $\MM$, its distribution with respect to $\tau$ is
the linear form on the free algebra generated by elements indexed by $I$
that maps a polynomial $P((X_i)_{i \in I})$ in non commuting variables to
$\tau(P\big((x_i)_{i\in I}\big)$. Its $*$-distribution is the distribution
of $(x_i, x_i^*)_{i \in I}$. The fact that two families of \emph{bounded}
operators have the same $*$-distributions is known to be equivalent to
saying that they generate isomorphic tracial von Neumann algebras (Lemma
\ref{thm=lemme_equidistribution_implique_isomorphisme}).

\bigskip

The main result of this paper is the following theorem:
\begin{thm}
\label{thm=thm_principal} 
Let $(\MM,\tau)$ and $(\NN, \widetilde \tau)$ be von Neumann algebras (on
some Hilbert space $H$) equipped with faithful normal finite normalized
traces. Let $E \subset L_p(\MM,\tau)$ be a subspace of $L_p(\MM,\tau)$, and
let $u:E\rightarrow L_p(\NN,\widetilde \tau)$ be a linear map. Denote by
$\id \otimes u: \Mn \otimes E \rightarrow \Mn \otimes L_p(\widetilde \tau)$
the natural extension of $u$ to $\M{n}{E}$. Fix $0<p<\infty$ such that $p$
is not an even integer.

Assume that the following boundedness condition holds: $E \subset
L_\infty(\MM) (= \MM)$.

Assume that for all $n \in \N$ and all $X \in \M{n}{E}$, the following
equality between the $p$-``norms'' holds:
\begin{equation}
\label{eq=hypothese_p_c.isometrie}
\|\un_n \otimes \un_\MM + X\|_p = \|\un_n \otimes \un_\NN + (\id\otimes
u)(X)\|_p.
\end{equation}

Let $VN(E)$ denote the von Neumann subalgebra generated by $E$ in
$\MM$. Then $u(E) \subset L_\infty(\NN)$ extends to a von Neumann algebra
isomorphism $u:VN(E) \rightarrow VN(u(E))$ that preserves the traces, and
this extension is unique.

In particular, if $E$ is an algebra, then $u$ agrees with the
multiplicative structure of $E$: if $x,y \in E$, then
$u(xy)=u(x)u(y)$. Moreover, if $x \in E$ and $x^* \in E$, then $u(x^*)
= u(x)^*$.
\end{thm}

First some remarks: as in the commutative case, the condition $p
\notin 2 \N$ is crucial. Indeed in the simplest case when $p = 2 n$
and $E = \C X$ is one-dimensional, with $X^*=X$ and $Y = u(X) = Y^*$,
then condition \eqref{eq=hypothese_p_c.isometrie} holds as soon as the
distributions of $X$ and $Y$ coincide on every polynomial of degree
less than $2n$, which does not imply that the distributions agree on
every polynomial. 

It is also easy to see that it is necessary to allow matrix
coefficients to appear in \eqref{eq=hypothese_p_c.isometrie}, and that
the theorem does not hold when \eqref{eq=hypothese_p_c.isometrie} is
assumed only for $x \in E$. A simple example is when $E = \MM = \NN=
\Mn$ equipped with its normalized trace $\tr_n$ and $u$ is the
transposition map $u:(a_{ij}) \rightarrow (a_{ji})$. Then $u$ is
isometric for every $p$-norm but is not a morphism of algebras.

However, it is unclear whether the theorem holds if
\eqref{eq=hypothese_p_c.isometrie} is only assumed for every $x \in
\M{n}{E}$ for a \emph{fixed} $n$ (even for $n=2$).

When $p = 2m$ is an even integer, the situation is different: it is
possible to show that if \eqref{eq=hypothese_p_c.isometrie} holds for
$n=m$, then
\eqref{eq=hypothese_p_c.isometrie} holds for any $n$. See Theorem
\ref{thm=p=2m_m_isometrie_implique_c.isom}.
\bigskip

The techniques used in the proof of Theorem \ref{thm=thm_principal} do not
allow to state the result when $E$ is a general subspace of $L_p(\MM)$
(\ie{} not necessarily made of bounded operators). Indeed the proof
relies on Lemma \ref{thm=lemme_equidistribution_implique_isomorphisme}
which says that the $*$-distribution of a family of bounded operators
characterizes the von Neumann algebra they generate. This result is known
to be false for unbounded operators even in the commutative case (it is the
moment problem). Moreover the proof relies on the expansion in power series
of operators of the form $|1+x|^p$ which allows to compute the
$*$-distribution of operators (Lemma
\ref{thm=egalite_entre_normeP_et_trace_avec_adjoint}). At first sight this
seems to require that the operator $x$ is bounded. However it is possible
to get some results of this kind for unbounded operators using a \nc{}
version of dominated convergence theorem from \cite{MR840845}: see Lemma
\ref{thm=egalite_entre_normeP_et_trace_avec_adjoint_nb}. It is also
immediate to see that Theorem \ref{thm=thm_principal} still holds if the
boundedness condition is replaced by the assumption that $E \cap
L_\infty(\MM)$ (or even $E \cap L_\infty + u^{(-1)}\left(u(E) \cap
  L_\infty\right)$ by Theorem
\ref{thm=equivalence_entre_dans_L2n_et_image}) is dense in $E$.

In the case when $E$ is self-adjoint and $u$ is assumed to map a
self-adjoint operator to a self-adjoint operator (which is \emph{a
posteriori} always true, see Lemma
\ref{thm=isometrie_preserve_adjoint}), Theorem \ref{thm=thm_principal}
can be deduced from the commutative Theorem
\ref{thm=theoreme_1_rudin}. Although it is contained in the general
case, this special case is proved in the first section of this paper,
since the proof uses the same idea as in the general case but with
simpler computations.

In the second section of this paper the main technical results are
proved. The first one establishes the link between the trace of products of
operators and $p$-norms of linear combinations of these operators (Lemma
\ref{thm=egalite_entre_normeP_et_trace_avec_adjoint} for bounded operators
and Lemma \ref{thm=egalite_entre_normeP_et_trace_avec_adjoint_nb} for the
general case). The second one (Theorem
\ref{thm=equivalence_entre_dans_L2n_et_image}) proves that in the setting
of Theorem \ref{thm=thm_principal}, if $E \subset L_\infty(\MM)$ then $u(E)
\subset L_\infty(\NN)$.
 
In section \ref{section=preuve_resultat_principal} the main theorem
(analogous to Theorem \ref{thm=theoreme_1_rudin}) is derived from
Lemma \ref{thm=egalite_entre_normeP_et_trace_avec_adjoint} (Theorem
\ref{thm=equirepartition} and Theorem \ref{thm=thm_principal}) and
also reformulated in the operator space setting (Corollary
\ref{thm=deuxieme_corollaire}). We also derive an approximation result
and we discuss the necessity of taking matrices of arbitrary size in
\eqref{eq=hypothese_p_c.isometrie} (but this question is mainly left
open).

In a last part, some other consequences of the results of section
\ref{part=resultats_techniques} are established, dealing with maps
defined on subspaces of $L_p$ which have an additional algebraic
structure (e.g. self-adjoint, or stable by multiplication...). In
particular a \nc{} analogue of Rudin's Theorem
\ref{thm=theoreme_2_rudin} is derived. We end the paper with some
comments and questions.

\section{Self-adjoint case}
In this section we prove the special case explained in the
introduction as a consequence of the commutative theorem.

Let $p \in \R^+ \setminus 2\N$, $E \subset \MM$ and $u:E\rightarrow \NN$ be
as in Theorem \ref{thm=thm_principal}. Assume furthermore that $E$ is
self-adjoint (if $x \in E$, $x^* \in E$) and that $u(x^*) = u(x)^*$ for $x
\in E$.
 
Let us sketch the proof in this special case: for any self-adjoint
operators $x_1,\dots x_n$ in $E$, denote $y_k = u(x_k)$. Then for any
self-adjoint matrices $a_1,\dots a_n$, since $\sum_k a_k \otimes x_k$ and
$\sum_k a_k \otimes y_k$ are self-adjoint, they generate commutative von
Neumann algebras, and Rudin's theorem can be applied to deduce that they
have the same distribution. The conclusion thus follows from Lemma
\ref{thm=lemme_equidistribution_implique_isomorphisme} and from the
following linearization result (and the fact that $E$ is spanned by
self-adjoint operators):

\begin{lemma}
\label{thm=cas_auto_adjoint}
Let $x_1,\dots x_n \in \MM$ and $y_1,\dots y_n \in \NN$ be self-adjoint
operators. Assume that for all $m$ and all self-adjoint $m \times m$
matrices $a_1 \dots a_n$, the operators $a_1 \otimes x_1 + \dots a_n
\otimes x_n$ and $a_1 \otimes y_1 + \dots a_n \otimes y_n$ have the
same distribution with respect to the traces $\tr_m \otimes \tau$ and
$\tr_m \otimes \widetilde \tau$:
\begin{equation}
\label{eq=hypo_egalite_des_distribs}
\dist(a_1 \otimes x_1 + \dots a_n \otimes x_n) = \dist (a_1 \otimes y_1 +
\dots a_n \otimes y_n)
\end{equation}

Then $(x_1,\dots x_n)$ and $(y_1,\dots y_n)$ have the same
distribution.
\end{lemma}
Independently of our work, this lemma was obtained in \cite{A_CITER} using
random matrices, and was used to give a new formulation of Connes's
embedding problem.

Here we provide a different and elementary proof that consists in
exhibiting specific matrices $a_1,\dots a_n$. The idea is the same as
in the proof of the general case of Theorem \ref{thm=thm_principal},
but here the computations are simpler.

In fact the result of Collins and Dykema is apparently slightly
stronger than the one stated above in the sense that they only assume
that \eqref{eq=hypo_egalite_des_distribs} holds for any self-adjoint
matrices $a_i$ with a spectrum included in $[c,d]$ for some fixed real
numbers $c<d$. But it is not hard to deduce their result from the one
above. More precisely, $m \in \N$ and $c<d$ being fixed, if one only
assumes that \eqref{eq=hypo_egalite_des_distribs} holds for any
self-adjoint matrices $a_i$ of size $m$ with a spectrum included in
$[c,d]$, then it holds for any self-adjoint matrices $a_i \in \Mm$
(without restriction on the spectrum). Indeed, if $c<\lambda<d$ and
$a_i \in \Mm$ are arbitrary, then for $t \in \R$ small enough, the
matrices $\lambda \un_m + t a_i$ all have spectrum in $[c,d]$; and the
distribution of $\sum (\un_m + t a_i) \otimes x_i$ for infinitely
many different values of $t$ is enough to determine the distribution
of $\sum a_i \otimes x_i$.
\begin{proof}[Proof of Lemma \ref{thm=cas_auto_adjoint}]
Let $m$ be an integer and take $(i_1,\dots i_m) \in \left\{1,2,\dots
n\right\}^m$. We want to prove that
\[ \tau(x_{i_1} x_{i_2} \dots x_{i_n}) = \widetilde 
\tau( y_{i_1} y_{i_2} \dots y_{i_n}).
\]

Relabeling and repeating if necessary the $x_i$'s and $y_i$'s, it is
enough to prove it when $m=n$ and $i_k=k$ for all $k$. We are left to
prove that
\begin{equation}
\label{eq=egalite_des_moments_pour_collins_dykema}
 \tau(x_1 x_2 \dots x_{n}) = \widetilde \tau( y_1 y_2 \dots y_n).
\end{equation}

Take $n$ complex numbers $z_1, \dots z_n$ and consider the $n \times
n$ self-adjoint matrices $a_k= z_k e_{k,k+1} + \overline{z_k}
e_{k+1,k}$ if $k<n$ and $a_n = z_n e_{n,1} + \overline{z_n} e_{1,n}$;
the expression $(\tr_n\otimes \tau) \left((\sum_{k=1}^n a_k \otimes
x_k)^n\right)$ can be viewed as a polynomial in the $z_j$'s and the
$\overline{z_j}$'s, and the coefficient in front of $z_1 z_2 \dots
z_n$ is equal to $\tau(x_1 x_2 \dots x_n)$. This is not hard to check
from the trace property of $\tau$ and from the fact that for a
permutation $\sigma$ on $\{1;2\dots n\}$, $\tr_n (e_{\sigma(1),\sigma(1)
+1 \mod n} e_{\sigma(2),\sigma(2) +1 \mod n} \dots
e_{\sigma(n),\sigma(n) +1\mod n})$ is nonzero if and only if $\sigma$
is a circular permutation, in which case it is equal to $1/n$.

Thus \eqref{eq=egalite_des_moments_pour_collins_dykema} holds, and
this concludes the proof.
\end{proof}

\begin{rem}
\label{rem=propriete_combinatoire_des_eij}
The following property of $n$-uples of $m \times m$ matrices $a_1, \dots
a_n \in \Mm$ is the key combinatorial property used in the proof above and
will later on be considered in this paper:
\begin{equation}
\label{eq=propriete_combinatoire_des_eij}
\tr_m(a_{\sigma(1)}a_{\sigma(2)}\dots a_{\sigma(n)}) = \left\{
\begin{array}{ll}
1 & \textrm{for a circular permutation $\sigma$ on }\left\{1;2;\dots
n\right\}.\\ 0 & \textrm{for another permutation $\sigma$.}
\end{array}
\right.
\end{equation}
A permutation $\sigma$ is said to be circular if there is an integer
$k$ such that $\sigma(j) = j + k \mod n$ for all $1\leq j \leq n$.

As noted in the proof above (and it was the main combinatorial trick in the
proof), the matrices $a_j = n^{1/n} e_{j,j+1 \mod n} \in \Mn$ have the
property \eqref{eq=propriete_combinatoire_des_eij}. But in fact in section
\ref{part=resultats_autres}, it will be interesting to find $n$ matrices
$a_1 \dots a_n$ with the same property but with smaller size. And this is
possible with matrices of size $m$ for $m \geq n/2$:

If $n = 2m$, then the following choice of the $a_j \in
\Mm\ (j=1\dots n)$ works:
\begin{eqnarray*}
a_{2j-1}  = e_{j,j} & \textrm{for }j=1\dots m\\
a_{2j}  = e_{j,j+1} & \textrm{for }j=1\dots m-1\\
a_{2m} = m e_{m,1}
\end{eqnarray*}
If $n = 2m-1$, then the following choice of the $a_j \in \Mm \
(j=1\dots n)$ works:
\begin{eqnarray*}
a_{2j-1} = e_{j,j} & \textrm{for }j = 1\dots m-1\\
a_{2j}= e_{j,j+1} & \textrm{for }j=1\dots m-1\\
a_{2m-1} = m e_{m,1}.&
\end{eqnarray*}
 \end{rem}

\section{Expression of the moments in term of the $p$-norms}
\label{part=resultats_techniques}
In this section, we prove that the trace of a product of finitely many
operators or of their adjoints can be computed from the $p$-norm of the
linear (matrix-valued) combinations of these operators. The main results
are Lemma \ref{thm=egalite_entre_normeP_et_trace_avec_adjoint} for bounded
operators and its refinement Lemma
\ref{thm=egalite_entre_normeP_et_trace_avec_adjoint_nb} for unbounded
operators. We also prove that a map $u$ as in Theorem
\ref{thm=thm_principal} maps a bounded operator to a bounded operator
(Theorem \ref{thm=equivalence_entre_dans_L2n_et_image}).

\subsection{Case of bounded operators}

First suppose we are given $x_1, x_2, \dots x_n$ elements of the von
Neumann algebra $\MM$ (here the $x_i$'s are \emph{bounded operators}),
and $\varepsilon_1, \dots \varepsilon_n \in \left\{1,*\right\}$. If
$x$ is an element of a von Neumann algebra and $\varepsilon \in
\left\{1,*\right\}$, let $x^\varepsilon$ denote $x$ if $\varepsilon =
1$ and $x^*$ if $\varepsilon = *$ (for a complex number $z$, $z^* =
\overline z$). 

For clarity, we will adopt the following (classical) notation: if $z =
(z_1,\dots z_n) \in \C^n$ and $k = (k_1,\dots k_n) \in \N^n$, we write $z^k
= \prod_j z_j^{k_j}$ and ${\bar z}^k =\prod_j {\bar z_j}^{k_j}$. In the
same way, one writes $z^\varepsilon =\prod_j z_j ^{\varepsilon_j}$.  If $f$
is a formal series $f(z) = \sum_{k,l \in \N^n} a_{k,l} z^k{ \bar z}^l$, we
denote $f(z)[z^k {\bar z}^l] = a_{k,l}$. 

Pick $n$ matrices $a_1,\dots a_n$ with complex coefficients (say of size
$m$). The $a_j$'s will soon be assumed to satisfy
\eqref{eq=propriete_combinatoire_des_eij}. For all $z = (z_1, z_2, \dots
z_n) \in \C^n$, denote by $S_z \in \M{m}{\MM}$ the matrix
\begin{equation}
\label{eq=definition_de_Sz}
S_z = S(z_1,\dots z_n) = \un + \sum_{j=1}^n z_j a_j^{\varepsilon_j}
\otimes x_j.
\end{equation}

The following combinatorial lemma justifies the choice of the $a_j's$:
\begin{lemma}
\label{thm=lemme_combinatoire}
Denote by $\alpha(\varepsilon)$ or simply $\alpha$ the number of indices $1
\leq j \leq n$ such that $\varepsilon_j =*$ and $\varepsilon_{j+1} =1$
(again if $j=n$, $\varepsilon_{n+1} = \varepsilon_1$).

If the $a_j$'s satisfy \eqref{eq=propriete_combinatoire_des_eij} and $S_z$
is defined by \eqref{eq=definition_de_Sz}, then for any integer $k$,
\begin{equation}
\label{eq=enonce_lemme_combinatoire}
\tau^{(m)} \left( (S_z^* S_z-1)^k [z^\varepsilon]\right) =
\tau(x_1^{\varepsilon_1} x_2^{\varepsilon_2}\dots x_n^{\varepsilon_n })
k\binom{\alpha}{n-k}
\end{equation}
\end{lemma}
\begin{proof}
Recall that
\[S_z^* S_z - \un = \sum_{j \leq n} z_j a_j^{\varepsilon_j} \otimes x_j 
 + \sum_{j \leq n} \overline{z_j} a_j^{\varepsilon_j *} \otimes x_j^*
+\sum_{i,j \leq n} \overline{z_i} z_j a_i^{\varepsilon_i *}
a_j^{\varepsilon_j} \otimes x_i^* x_j .\]

For one of the terms of $\sum_{j \leq n} z_j
a_j^{\varepsilon_j}\otimes x_j$ to bring a contribution to the
coefficient of $\prod_j z_j ^{\varepsilon_j}$ in $(S_z^* S_z -
\un)^k$, it is necessary that $\varepsilon_j=1$, and then $ z_j
a_j^{\varepsilon_j}\otimes x_j = z_j^{\varepsilon_j} a_j \otimes
x_j^{\varepsilon_j}$.  In the same way, for one of the terms of
$\sum_{j \leq n} \overline{z_j} a_j^{\varepsilon_j * }\otimes
x_j^* $ to bring a contribution, it is necessary that
$\varepsilon_j=*$ and then $ \overline{z_j} a_j^{\varepsilon_j
*}\otimes x_j^* = z_j^{\varepsilon_j} a_j \otimes
x_j^{\varepsilon_j}$. Last, for one of the terms of $\sum_{i,j \leq
n} \overline{z_i} z_j a_i^{\varepsilon_i *} a_j^{\varepsilon_j}
\otimes x_i^* x_j$ to have a nonzero contribution, the values of
$\varepsilon_i$ and $\varepsilon_j$ must be $\varepsilon_i=*$ and
$\varepsilon_j=1$, and then $\overline{z_i} z_j a_i^{\varepsilon_i
*} a_j^{\varepsilon_j} \otimes x_i^* x_j = z_i^{\varepsilon_i}
z_j^{\varepsilon_j} a_i a_j \otimes x_i^{\varepsilon_i}
x_j^{\varepsilon_j}$. Thus if one denotes $y_j = x_j^{\varepsilon_j}$,
\begin{multline*}
  \tau^{(m)}\paren{\paren{S_z^*S_z-\un}^k} [z^\varepsilon] = \\
   \tau^{(m)}\paren{\paren{\sum_{1\leq j \leq n} z_j^{\varepsilon_j}
    a_j \otimes y_j + \!\!\!\!\!\!\!\!
   \sum_{ i,j , \varepsilon_i=*\textrm{ and } \varepsilon_{j}=1} 
   \!\!\!\!\!\!\!z_i^{\varepsilon_i} z_j^{\varepsilon_j}
    a_i a_j \otimes y_i y_j }^k} [z^\varepsilon].
\end{multline*}

Developing and using the assumption
\eqref{eq=propriete_combinatoire_des_eij} on the $a_j$'s, one gets
\begin{equation}
\label{eq=valeur_de_gamma_n}
\tau^{(m)}\paren{\paren{S_z^*S_z-\un}^k} [z^\varepsilon] = \sum_{l=1}^n
C_l \tau\paren{ y_l y_{l+1} \dots y_{l-1}},
\end{equation} 
where the indices have to be understood modulo $n$ and where $C_l$ denotes
the number of ways of writing formally the word $y_l y_{l+1} \dots y_{l-1}$
(which is of length $n$) as a concatenation of $k$ ``elementary bricks'' of
the form $y_j$ (for $1\leq j \leq n$) or $y_j y_{j+1}$ with
$\varepsilon_j=*$ and $\varepsilon_{j+1}=1$. Each of these bricks has
length $1$ or $2$. If $\alpha_l$ denotes the number of apparitions of the
subsequence $*,1$ in the sequence $\varepsilon_l,\varepsilon_{l+1 \mod n},
\dots \varepsilon_{l-1 \mod n}$ (not cyclically this time!), then for $C_l$
to be non zero it is necessary that $k \leq n \leq k+\alpha_l$. In that
case $C_l$ is equal to the number of ways of choosing the $n-k$ bricks of
size $2$ among the $\alpha_j$ possible, the other bricks being of size
$1$. Thus $C_l = \binom{\alpha_l}{n-k}$. The fact that $\tau$ is a trace
then allows to write \eqref{eq=valeur_de_gamma_n} as
\begin{eqnarray*}
\tau^{(m)}\paren{\paren{S_z^*S_z-\un}^k} [z^\varepsilon]&=&
\tau\paren{y_1 y_2 \dots y_n} \sum_l \binom{\alpha_l}{n-k}\\ &=&
\tau\paren{x_1^{\varepsilon_1} x_2^{\varepsilon_2}\dots x_n^{\varepsilon_n
}} \sum_l \binom{\alpha_l}{n-k}.
\end{eqnarray*}

It remains to notice that $\alpha_l = \alpha -1$ if $\varepsilon_{l-1}=*$
and $\varepsilon_{l}=1$ (which is the case for $\alpha$ different values of
$l$), and that $\alpha_l = \alpha$ otherwise (for the $n-\alpha$ remaining
values of $l$). The preceding equation then becomes
\begin{equation*}
\tau^{(m)}\paren{\paren{S_z^*S_z-\un}^k} [z^\varepsilon] =
  \tau\paren{x_1^{\varepsilon_1} x_2^{\varepsilon_2}\dots
  x_n^{\varepsilon_n }} \paren{ \alpha\binom{\alpha-1}{n-k} + (n-\alpha)
  \binom{\alpha}{n-k}}.
\end{equation*}

Equation \eqref{eq=enonce_lemme_combinatoire} follows from the elementary
equality
\[\alpha\binom{\alpha-1}{n-k} + (n-\alpha) \binom{\alpha}{n-k} =
k\binom{\alpha}{n-k}.\]
\end{proof}
Note that the above proof only uses combinatorial arguments, it therefore
also holds with minor modifications when the assumption $x_j \in
\mathcal M$ is replaced by $x_j \in L_n(\MM, \tau)$ for all $j$.

The following lemma establishes the link between the $p$-norm of $S_z$
and the trace of the product of the $x_j^{\varepsilon_j}$.
\begin{lemma}
\label{thm=egalite_entre_normeP_et_trace_avec_adjoint}
Let $0<p<\infty$. Let $a_1,\dots a_n$ be matrices \emph{satisfying
  \eqref{eq=propriete_combinatoire_des_eij}}, and, remembering
\eqref{eq=definition_de_Sz}, define the function $\varphi: \R^n \rightarrow
\C$ by
\[\varphi(r_1, \dots r_n) = \frac{1}{(2\pi)\sp n} \int_{[0,2\pi]^n} \!\!
\left\|S\paren{ r_1 e^{i \theta_1}, \dots r_n e^{i \theta_n}}\right\|_p^p
    \prod_j \exp(-i \theta_j)^{\varepsilon_j}  d \theta_1 \dots d
      \theta_n.\]

Then $\varphi$ is indefinitely differentiable on a neighborhood of $0$, and
if $\alpha$ is defined as in Lemma \ref{thm=lemme_combinatoire}
\begin{multline}
\label{eq=egalite_entre_normeP_et_trace_avec_adjoint}
\frac{d^{(n)}}{d r_1 \dots d r_n} \varphi (0,\dots 0) = \lim_{r_1,
  \dots r_n \rightarrow 0} \frac{1}{r_1 \dots r_n }\varphi(r_1, \dots
  r_n) = \\ \tau(x_1^{\varepsilon_1} x_2^{\varepsilon_2}\dots
  x_n^{\varepsilon_n }) \sum_{k=0}^\alpha (n-k) \binom{\nicefrac p
  2}{n-k}\binom{\alpha}{k}.
\end{multline}
\end{lemma}
\begin{proof}
The idea of the proof is the following: $S_z$ is a small perturbation
of the unit, which allows to write $\left|S_z\right|^p$ as a
converging series. Equation
\eqref{eq=egalite_entre_normeP_et_trace_avec_adjoint} follows from the
identification of the term in front of $\prod_j
z_j^{\varepsilon_j}$. First write:

\begin{eqnarray*}
S_z^* S_z & = &\un + \sum_{j \leq n} z_j
a_j^{\varepsilon_j}\otimes x_j + \sum_{j \leq n} \overline{z_j}
a_j^{\varepsilon_j *} \otimes x_j^* +\sum_{i,j \leq n}
\overline{z_i} z_j a_i^{\varepsilon_i *} a_j^{\varepsilon_j}
\otimes x_i^* x_j\\ & = & \un + \sum_{1\leq j\leq n^2 + 2n} C_j.
\end{eqnarray*}

In the last line, we denoted by $C_j$ the $n^2 + 2n$ terms that appear on
the preceding line. Remark that if $\sup |z_j| = \delta \leq 1$, then
$\|C_j\| \leq \delta K$ where $K =
\max_j\paren{\|a_j\|\|x_j\|,\|x_j\|^2\|x_j\|^2}$.

By the functional calculus for bounded operators, for $z$ small enough
(\ie{} $\|\un - S_z^* S_z\| < 1$), one has:
\begin{equation}
\label{eq=valeur_de_S^p}
(S_z^* S_z)^{\nicefrac p 2} = \sum_{k \geq 0} \binom{\nicefrac p 2}{k}
\paren{S_z^* S_z -\un}^k = \sum_{k \geq 0} \binom{\nicefrac p 2}{k}
\sum_{1\leq j_1, \dots j_k \leq n^2 + 2n} C_{j_1} \dots C_{j_k}.
\end{equation}

In this equality, $\binom{\nicefrac p 2}{n}$ is the generalized binomial
coefficient defined, for $\beta \in \C$ and $n \in \N$, by: 
\begin{equation}
\label{eq=def_coeff_binomial}
\binom{\beta}{n} = \beta (\beta -1) \dots (\beta - n+1)/n!.
\end{equation}

The series above converges absolutely and uniformly when $\delta = \sup
|z_j|$ is small enough, \ie{} $K (n^2+2n) \delta < 1$. Indeed, $\|C_{j_1}
\dots C_{j_k}\| \leq \delta^k K^k$, and since $\binom{\nicefrac p 2}{k}$
tends to $0$ as $k \rightarrow \infty$, one has
\[\sum_{k \geq 0} \sum_{1\leq j_1, \dots j_k \leq n^2 + 2n}
\sup_{|z_j|\leq \delta \forall j}\left\|\binom{\nicefrac p 2}{k}
  C_{j_1} \dots C_{j_k}\right\| \leq \sum_{k \geq 0} (n^2 + 2n)^k
\left|\binom{\nicefrac p 2}{k}\right| \delta^k K^k < \infty.\]

We can thus reorder the terms of the sum \eqref{eq=valeur_de_S^p}
along powers of $z_j$ and $\overline z_j$: 
\begin{equation}
\label{eq=_somme_Sz_p_reorganisee}
 |S_z|^p = \sum_{ k, l \in \N^n} \, z_1^{k_1} \dots z_n^{k_n}
  {\bar z_1}^{l_1} \dots {\bar z_n}^{l_n} D_{k,l} ,
\end{equation}
where $D_{ k, l}$ are some operators in $\Mm \otimes \MM$, which are in fact
some polynomials in $a_1^{\varepsilon_1}\otimes x_1 \dots
a_n^{\varepsilon_n} \otimes x_n$ and their adjoints. Taking the trace
$\tau^{(m)}$ on both sides of \eqref{eq=_somme_Sz_p_reorganisee} , one gets

\begin{equation}
\label{eq=norme_Sz_reorganisee}
 \|S_z\|_p^p = \sum_{ k,  l \in
  \N^n} \lambda_{ k,  l}\, z_1^{k_1} \dots z_n^{k_n}
{\bar z_1}^{l_1} \dots {\bar z_n}^{l_n}.
\end{equation}

In this sum, we wrote $ k=(k_1, \dots k_n)$ and $ l = (l_1, \dots
l_n)$. The coefficient $\lambda_{ k, l}$ is equal to
\[ \lambda_{k,l} = \sum_{j \leq |l|+|k|} \binom{p/2}{j}
\tau^{(m)}\paren{\paren{S_z^*S_z-\un}^j} [z^k {\bar z}^l]. \]

If $E$ is defined as the set of indices $(k,l) \in \N^n \times \N^n$
such that $k_j - l_j = 1$ if $\varepsilon_j=1$ and $k_j - l_j = -1$ if
$\varepsilon_j=*$, then for $r_1, \dots r_n$ small enough, we are
allowed to exchange the series and the integral in the definition of
$\varphi(r_1,\dots r_n)$ and we get the following expression of $\varphi$
as a converging power series:
\[\varphi(r_1,\dots r_n) = \sum_{(k,l) \in E} \lambda_{k,l}\, r_1^{k_1+l_1}
\dots
r_n^{k_n+l_n}.\]

The two left-hand sides of
\eqref{eq=egalite_entre_normeP_et_trace_avec_adjoint} are thus equal
to $\lambda_{k^0,l^0}$ where $k^0_j = 1$ if $\varepsilon_j=1$, $k^0_j =
0$ else, and $l^0_j = 1 - k^0_j$. In other words, $\lambda_{k^0,l^0}$ is
the coefficient of $\prod_j z_j ^{\varepsilon_j}$ in
\eqref{eq=norme_Sz_reorganisee}:
\begin{equation}
\label{eq=egalite_entre_phi_et_bkl}
\frac{d^{(n)}}{d r_1 \dots d r_n} \varphi (0,\dots 0) = \lim_{r_1, \dots
  r_n \rightarrow 0} \frac{1}{r_1 \dots r_n }\varphi(r_1, \dots r_n) =
  \lambda_{k^0,l^0},
\end{equation}
with
\begin{equation}
\label{eq=valeur_de_bkl}
  \lambda_{k^0,l^0} = \sum_{j\in \N} \binom{\nicefrac p
    2}{j}\underbrace{\tau^{(m)}\paren{\paren{S_z^*S_z-\un}^j}
    [z^\varepsilon]}_{\egdef \gamma_j}.
\end{equation}

But from Lemma \ref{thm=lemme_combinatoire}, 
\[\gamma_j = \tau\paren{x_1^{\varepsilon_1} x_2^{\varepsilon_2}\dots
  x_n^{\varepsilon_n }} j\binom{\alpha}{n-j}.\]

Putting this equation together with \eqref{eq=egalite_entre_phi_et_bkl} and
\eqref{eq=valeur_de_bkl}, we finally get
\eqref{eq=egalite_entre_normeP_et_trace_avec_adjoint}, which proves the
Lemma.
\end{proof}
\begin{rem}
\label{rem=ppair_on_a_meme_chose}
In the case when $p/2$ is an integer (\ie{} $p$ is an even
integer), the same result holds in a more general setting, when the
$x_i$'s are not bounded but are in the \nc{} $L_p$ space
associated to $(\MM,\tau)$. Indeed, then the sum on the right-hand side
of \eqref{eq=valeur_de_S^p} makes sense as a finite sum of elements
which all are in $L_1(\MM,\tau)$. Indeed, from H{\"o}lder's inequality, a
product of $k$ elements of $L_p$ is in $L_{p/k}$. This allows to take
the trace in \eqref{eq=valeur_de_S^p} and to follow the rest of the
proof.

Of course when $p$ is different from an even integer, the proof does
not apply for unbounded operators: it is indeed unclear what sense
should be given to the series \eqref{eq=valeur_de_S^p}, and more
importantly taking the trace to get \eqref{eq=norme_Sz_reorganisee}
makes no sense. However, using a \nc{} dominated convergence theorem
from \cite{MR840845}, it is possible to modify the proof and get
similar results with unbounded operators.
\end{rem}

\subsection{Unbounded case}
\label{part=non_borne}
The reader is refered to \cite{MR840845} for all facts and definitions on
measure topology and generalized $s$-numbers. Just recall that if
$(\MM,\tau)$ is a von Neumann algebra with a \nff{} normalized trace, the
$t$-th singular number of a closed densely defined (possibly unbounded)
operator $Y$ affiliated to $\MM$ is defined as
\[\mu_t(Y) = \inf\left\{ \|YE\|, E \textrm{ is a projection in $\MM$ with } 
  \tau(\un - E) \leq t\right\}.\] 
The map $t \rightarrow \mu_t(Y)$ is non-increasing and vanishes on $t \geq
1$, and $\mu_t(Y)<\infty$ if $t>0$.

Moreover the measure topology makes the set of $\tau$-measurable operators
affiliated to $\MM$ a topological $*$-algebra in which a sequence $(Y_n)$
converges to $Y$ if and only if $\mu_t(Y-Y_n) \rightarrow 0$ for all
$t>0$. More precisely, the following inequalities hold for any positive
real numbers $s,t>0$ and any (closed densely defined) operators $T$ and $S$
affiliated with $\MM$ (Lemma 2.5 in \cite{MR840845}):
\begin{eqnarray}
\mu_t(\lambda T) &=& |\lambda| \mu_t(T) \textrm{ for any }\lambda \in \C\\
\label{eq=inegalite_sing_num_adjoint}
\mu_t(T) &=& \mu_t(|T|) = \mu_t(T^*)\\
\label{eq=inegalite_sing_num_add}
\mu_{t+s}(T+S)& \leq & \mu_t(T) + \mu_s(S)\\
\label{eq=inegalite_sing_num_mult}
\mu_{t+s}(TS) &\leq & \mu_t(T) \mu_s(S).
\end{eqnarray}

Another property from \cite[Lemma 2.5]{MR840845} is the fact that
$\mu_s(f(T)) = f(\mu_s(T))$ for any operator $T \geq 0$ and any continuous
increasing function on $\R$ with $f(0)=0$. As a consequence, for any
continuous function $f$ on $\R$ with $f(0)=0$ and any self-adjoint $T$
affiliated with $\MM$, 
\begin{equation}
\label{eq=inegalite_sing_num_fonction}
\mu_t(f(T)) \leq \sup_{ |u| \leq \mu_t(T)} |f(u)|
\end{equation}

For any $0<p\leq\infty$, the \nc{} $L_p$-space $L_p(\MM,\tau)$ is identified
with the set of closed densely defined operators $Y$ affiliated to $\MM$ such
that the function $t \mapsto \mu_t(Y)$ is in $L_p([0,1], dt)$. Moreover,
the $p$-norm of this function is equal to $\|Y\|_p$.

We now fix $0<p<\infty$.

The first fact we prove is the following lemma, which basically says
that when the $x_j$'s are unbounded operators affiliated with $\MM$,
the development in power series of $|S_z|^p$
\eqref{eq=_somme_Sz_p_reorganisee} still holds, but in the measure
topology instead of the norm topology. 

\begin{lemma}
\label{thm=convergence_en_mesure_cas_NC}
Let $X$ be a closed densely defined operator affiliated with a von Neumann
algebra $(\MM, \tau)$. For $r>0$, denote by $Y_r$ the operator
\[Y_r = (\un+r X)^*(\un+r X) - \un = r X + r X^* + r^2 X^* X. \]

Then as $r \rightarrow 0$, the following convergence holds in the
measure topology:
\begin{equation}
\frac{1}{r^n}\left( |\un + rX|^p - \sum_{j=0}^n \binom{\nicefrac p 2}{j}
Y_r^j\right) \rightarrow 0.
\end{equation}
\end{lemma}
\begin{proof}%
We first claim that for all $t>0$, $\sup_{r<1} \mu_t(Y_r/r)<
\infty$. Indeed, from \eqref{eq=inegalite_sing_num_add}, we have:
\begin{eqnarray*} \mu_{2t}(Y_r/r) & \leq & \mu_t(X + X^*) + \mu_t(
r X^* X)\\
& = & \mu_t(X + X^*) + r \mu_t( X^* X)
\end{eqnarray*}
The claim follows from the fact that $\mu_t(x)<\infty$ for all
closed densely defined operator $x$.

Fix now $t>0$ and take $M = \sup_{r<1} \mu_t(Y_r/r)$. Then
(Proposition 2.2 in \cite{MR840845}) if $E = E_{[-M,M]}(Y_r/r)$ and
$r<1$, we have $\tau(\un - E) \leq t$ and by the functional
calculus, since $Y$ and $E$ commute and are self-adjoint,
\[ (1+Y_r)^{\nicefrac p 2}E = (E+Y_rE)^{\nicefrac p 2} E
   = \sum_{j\geq 0} \binom{\nicefrac p 2}{j} (Y_rE)^j E= \sum_{j\geq 0}
\binom{\nicefrac p 2}{j} Y_r^j E.\]
The previous series converges in the operator norm topology if $r
M<1$, since in that case, $\|Y_rE\|\leq r M <1$. Then
\begin{eqnarray*}
\left\|\frac{1}{r^n}\left( |S_{rz}|^p - \sum_{j=0}^n \binom{\nicefrac
p 2}{j} Y_r^j\right) E \right\| 
& = & \left\|\frac{1}{r^n}\left( (1+Y_r)^{\nicefrac p 2} -
\sum_{j=0}^n \binom{\nicefrac p 2}{j} Y_r^j\right) E \right\| \\ 
& = & \left\|\sum_{j \geq n+1} r^{-n} \binom{\nicefrac p 2}{j} 
(Y_rE)^j \right \|\\
& \leq & \sum_{j \geq n+1} \left|\binom{\nicefrac p 2}{j} \right|
r^{-n} (M r)^j \rightarrow 0.
\end{eqnarray*}

This proves that $\mu_t\left( r^{-n}\left( |\un + r X|^p - \sum_{j=0}^n
\binom{\nicefrac p 2}{j} Y_r^j\right)\right)$ tends to zero as $r
\rightarrow 0$ for every $t>0$. This concludes the proof.
\end{proof}

Let us denote by $Q_n$ the linear projection from the space of complex
polynomial $\C[r]$ to the subspace $\C_n[r]$ of the polynomials of degree
at most $n$ given by: $Q_n(r^k) = r^k$ if $r \leq n$ and $Q_n(r^k) = 0$ if
$k >n$. If $V$ is any vector space over the field of complex numbers, this
projection naturally extends to the space of polynomials with coefficients
in $V$ (this extension if simply the tensor product map $\id \otimes Q_n$
if one identifies the space of polynomials with coefficients in $V$ with
the tensor product $V \otimes \C[r]$). For simplicity this extension will
still be denoted by $Q_n$. The following result follows from Lemma
\ref{thm=convergence_en_mesure_cas_NC}:

\begin{corollaire} 
\label{thm=convergence_en_mesure_cas_NC2}
Let $X$ and $Y_r$ be as above (for any $r>0$). Then as  $r \rightarrow 0$,
\begin{equation*}
  \frac{1}{r^n}\left( |\un + rX|^p - Q_n\left(\sum_{j=0}^n \binom{\nicefrac p 2}{j}   Y_r^j\right) \right) \rightarrow 0.
\end{equation*}
\end{corollaire}
\begin{proof}
It follows immediately from Lemma \ref{thm=convergence_en_mesure_cas_NC}
and from the fact that if $T$ is affiliated with $(\MM,\tau)$ and $k>n$,
then 
\[ \frac{1}{r^n} r^k T \rightarrow 0 \textrm{ in the measure topology as }
r \rightarrow 0.\]
\end{proof}

The next step is to get a domination result necessary to apply Fack
and Kosaki's dominated convergence theorem. More precisely, we prove:
\begin{lemma}
\label{thm=domination_cas_NC}
With the same notation as above, there are constants $C$ and $K$ depending
only on $p$ and $n$ such that for all $r<1$ and all $0<t \leq 1$,
\begin{equation}
\label{eq=domination_pour_TCD_fackkosaki}
\mu_t \left( \frac{1}{r^n}\left( |\un + r X|^p - Q_n(\sum_{j=0}^n
      \binom{\nicefrac p 2}{j} Y_r^j)\right) \right) \leq 
C (\mu_{t/K}(X)^n + \mu_{t/K}(X)^p)
\end{equation}
\end{lemma}
\begin{proof}
Denote by $m_{t,r}$ the left-hand side of
\eqref{eq=domination_pour_TCD_fackkosaki}:
\[ m_{t,r} \egdef \mu_t \left( \frac{1}{r^n}\left( |\un + r X|^p -
    Q_n(\sum_{j=0}^n \binom{\nicefrac p 2}{j} Y_r^j)\right)
\right).\]

Fix an integer $K$ such that $K \geq 2n 2^{2n+1}$ and $K \geq 3
2^{2n+1}$. Define a real number $s$ by $s = t/K$.  To prove that $m_{t,r}
\leq C (\mu_s(X)^n + \mu_s(X)^p)$, we consider two cases, depending on the
value of $ r \mu_s(X)$.

First assume that $r \mu_s(X) \geq 1$.  

Note that there are some real numbers $\lambda_{k,\varepsilon}$ indexed by
the integers $k \geq 0 $ and $\varepsilon =
(\varepsilon_1,\varepsilon_2,\dots \varepsilon_k) \in \{1,*\}^k$ such that
for any $r>0$ (and any $n$),
\[ Q_n\left(\sum_{j=0}\sp n \binom{\nicefrac p 2}{j} Y_r^j\right) =
\sum_{k=0}^n \sum_{\varepsilon \in \{1,*\}^k} \lambda_{k,\varepsilon} r^k
X^{\varepsilon_1} X^{\varepsilon_2} \dots X^{\varepsilon_k}.\]

Thus, using \eqref{eq=inegalite_sing_num_add} $2^{n+1}$ times, one gets
\begin{eqnarray*}
  r^n m_{t,r} & \leq & \mu_{t/2^{n+1}}(|1+rX|^p) + \sum_{k=0}^n
  \sum_{\varepsilon \in \{1,*\}^k} |\lambda_{k,\varepsilon}| r^k
  \mu_{t/2^{n+1}}(  X^{\varepsilon_1}   X^{\varepsilon_2} \dots
  X^{\varepsilon_k})
\end{eqnarray*}

Since $t/2^{n+1} \geq t/K =s$, we have that
\begin{eqnarray*}
\mu_{t/2^{n+1}}(|1+rX|^p) &\leq& \mu_s(|1+rX|^p)\\
& = & \mu_s(|1+rX|)^p\\
& = & \mu_s(1+rX)^p\\
& \leq & (1+r\mu_s(X))^p\\
& \leq & 2^p (r\mu_s(X))^p\\
& \leq & \left\{
\begin{array}{cc} r^n 2^p \mu_s(X)^p & \textrm{if }p
    \geq n\\
r^n 2^p \mu_s(X)^n & \textrm{if }p \leq  n
\end{array} \right.
\end{eqnarray*}
In these computations, the fact that $\mu_s(f(T)) = f(\mu_s(T))$ for any
operator $T \geq 0$ and any continuous increasing function on $\R$ with
$f(0)=0$ was used, together with the assumption $1 \leq r \mu_s(X)$. 

From \eqref{eq=inegalite_sing_num_mult} and
\eqref{eq=inegalite_sing_num_adjoint}, we get, for $0 \leq k \leq n$,
\[\mu_{t/2^{n+1}}(X^{\varepsilon_1} X^{\varepsilon_2} \dots
X^{\varepsilon_k}) \leq \mu_{t/(k 2^{n+1})}(X)^k.\]
Since $t/(k 2^{n+1}) \geq t/K =s$, we have that
\[\frac 1 {r^n} r^k \mu_{t/2^{n+1}}(X^{\varepsilon_1} X^{\varepsilon_2} \dots
X^{\varepsilon_k}) \leq r^{k-n} \mu_s(X)^k \leq \mu_s(X)^n .\]

This concludes the proof that $m_{t,r} \leq C \left(\mu_s(X)^p +
  \mu_s(X)^n\right)$ for some $C$, in the case when $r \mu_s(X) \geq
1$.

Let us now assume that $r \mu_s(X) < 1$. We want to prove that in that
case, there is a constant $C$ not depending on $r$ and $t$ such that 
\begin{equation}
\label{eq=domination_a_prouver_cas_rmu_petit}
r^n m_{r,t} \leq C r^n \mu_s(X)^n.
\end{equation}

In the same way as above, write 
\begin{multline*} |\un + r X|^p - Q_n(\sum_{j=0}^n \binom{\nicefrac p 2}{j}
  Y_r^j) = \\ |\un + r X|^p - \sum_{j=0}^n \binom{\nicefrac p 2}{j}
  Y_r^j + \sum_{k=n+1}^{2n} \sum_{\varepsilon \in \{1,*\}^k} {\widetilde
    \lambda}_{k,\varepsilon} r^k X^{\varepsilon_1} X^{\varepsilon_2} \dots
  X^{\varepsilon_k},
\end{multline*}
for some real numbers ${\widetilde \lambda}_{k,\varepsilon}$ depending
neither on $r$ nor on $t$.

Again, using \eqref{eq=inegalite_sing_num_add}, one gets
\begin{multline*} 
  r^n m_{r,t} \leq \mu_{t/2^{2n+1}}\left(|\un + r X|^p - \sum_{j=0}^n
    \binom{\nicefrac p 2}{j} Y_r^j\right) +\\ \sum_{k=n+1}^{2n}
  \sum_{\varepsilon \in \{1,*\}^k} {\widetilde \lambda}_{k,\varepsilon} r^k
  \mu_{t/2^{2n+1}}(X^{\varepsilon_1} X^{\varepsilon_2} \dots
  X^{\varepsilon_k}).
\end{multline*}

The second term is easy to dominate using that $r \mu_s(X) < 1$ and that if
$n<k\leq 2n$, then $t/k2^{2n+1} \geq t/K =s$:
\[\mu_{t/2^{2n+1}}(X^{\varepsilon_1} X^{\varepsilon_2} \dots
  X^{\varepsilon_k}) \leq \mu_{t/k2^{2n+1}}(X)^k \leq \mu_s(X)^k \leq
  r^{n-k} \mu_s(X)^n.\]

For the first term, we use \eqref{eq=inegalite_sing_num_fonction} for $T =
Y_r$ and $f(x) = (1+x)^{p/2} - \sum_{k=0}^n \binom{\nicefrac p 2}{k} x^k$
(if $x \geq -1$, and say $f(x)=f(-1)$ else). Indeed, we have that $f(x) =
o(x^n)$ as $x \rightarrow 0$, in particular there is a constant $C_1$ such
that $|f(x)| \leq C_1 |x|^n$ if $|x| \leq 3$. If one proves that
$\mu_{t/2^{2n+1}}(Y_r) \leq 3 r \mu_s(X)$, we thus have that
\[\mu_{t/2^{2n+1}}\left(|\un + r X|^p - \sum_{j=0}^n
    \binom{\nicefrac p 2}{j} Y_r^j\right) \leq C_1 3^n r^n \mu_s(X)^n, \]
which would complete the proof of
\eqref{eq=domination_a_prouver_cas_rmu_petit}.

We are left to prove that $\mu_{t/2^{2n+1}}(Y_r) \leq 3 r \mu_s(X)$. But
since $t/2^{2n+1} \geq 3 t/K$, we have that $\mu_{t/2^{2n+1}}(Y_r) \leq
\mu_{3s}(Y_r)$, and thus using
 \eqref{eq=inegalite_sing_num_add}, one gets
 \begin{eqnarray*}
   \mu_{t/2^{2n+1}}(Y_r) \leq \mu_{3s}(r^2 X^* X + rX + rX^*) & \leq &
   r^2\mu_{s}(X^* X) + r \mu_s(X) + r \mu_s(X^*) \\
& = &(r\mu_r(X))^2 +2r \mu_s(X) \leq 3r\mu_s(X)
\end{eqnarray*}

\end{proof}

It is now possible to use Fack and Kosaki's dominated convergence
theorem \cite[Theorem 3.6]{MR840845} to prove the main result of
this part, which is the unbounded version of Lemma
\ref{thm=egalite_entre_normeP_et_trace_avec_adjoint}:

\begin{lemma}
\label{thm=egalite_entre_normeP_et_trace_avec_adjoint_nb}
Let $0<p<\infty$. Assume that $x_1 \dots x_n \in L_p(\MM,\tau)$ and take
$\varepsilon_1, \dots \varepsilon_n \in \left\{1,*\right\}$. If $x_j \in
L_{n}(\MM,\tau)$ for all $j$, then the trace $\tau(x_1^{\varepsilon_1} \dots
x_n^{\varepsilon_n})$ ``can be computed from the $p$-norms'' of linear
combinations of the $x_i$'s with coefficients in $\Mm$ for some $m$.

More precisely, let $m \in \N$ and $a_1,\dots a_n \in \Mm$ satisfying
\eqref{eq=propriete_combinatoire_des_eij} (as explained in Remark
\ref{rem=propriete_combinatoire_des_eij}, such a choice of the $a_j$'s
can be achieved for any $m \geq n/2$). Define $\alpha$ as in Lemma
\ref{thm=lemme_combinatoire} and denote $\forall z \in \C^n$
\[S_z = \un + \sum_{j=1}^n z_j a_j^{\varepsilon_j}\otimes x_j.\]

Let $Z = (Z_1,\dots Z_n)$ be a (classical) $\C^n$-valued random
variable where the $Z_j$'s are uniformly distributed in $\left\{\exp(2i
k \pi/3), k=1,2,3\right\}$ and independent. Denote by $\mathbb{E}$ the
expected value with respect to $Z$. Then

\begin{equation}
\label{eq=egalite_entre_normeP_et_trace_avec_adjoint_non_borne}
\frac{1}{r^n} \mathbb{E} \left[ \| S_{r Z}\|_p^p \prod_{j=1}^n
  \overline{Z_j}^{\varepsilon_j} \right] \xrightarrow{r \rightarrow 0}
  \tau(x_1^{\varepsilon_1} x_2^{\varepsilon_2}\dots x_n^{\varepsilon_n
  }) \sum_{k=0}^\alpha (n-k) \binom{\nicefrac p
  2}{n-k}\binom{\alpha}{k}.
\end{equation}
\end{lemma}
\begin{proof}
The first step of the proof consists in using Corollary
\ref{thm=convergence_en_mesure_cas_NC2} and Lemma
\ref{thm=domination_cas_NC} in the von Neumann algebra $(\M{m}{\MM},
\tau^{(m)})$ in order to apply \cite[Theorem 3.6]{MR840845}. Fix $z \in
\C^n$, and denote $Y_r = S_{rz}^* S_{rz} - \un$ (the dependence of $Y_r$
on $z$ is implicit). If $T_r = \frac{1}{r^n} \left( |S_{rz}|^p -
  Q_n\left(\sum_{j=0}^n \binom{\nicefrac p 2}{j} Y_r^j\right))\right)$,
then from Corollary \ref{thm=convergence_en_mesure_cas_NC2}, $T_r$ converges to
$0$ in measure, and from Lemma \ref{thm=domination_cas_NC}, $T_r$ is
dominated in the following way: there are positive constants $C$ and $K$
such that for any $0<t\leq 1$ and any $0<r<1$,
\begin{equation}
\mu_t(T_r) \leq C(\mu_{\nicefrac t K}(X)^p + \mu_{\nicefrac t K}(X)^n),
\end{equation}
where $X = \sum_{j=1}^n z_j a_j^{\varepsilon_j}\otimes x_j$. In
particular, $X \in L_p(\M{m}{\MM},\tau^{(m)})$ and $X \in
L_n(\M{m}{\MM},\tau^{(m)})$. To deduce that
\begin{equation}
\label{eq=consequence_du_tcd_non_commutatif}
 \frac{1}{r^n} \tau^{(m)} \left( |S_{rz}|^p - \sum_{j=0}^n 
 \binom{\nicefrac p 2}{j} Y_r^j\right) \rightarrow 0,
\end{equation}
it is thus sufficient to prove that the domination term $C(\mu_{\nicefrac t
  K}(X)^p + \mu_{\nicefrac t K}(X)^n)$ is (as a function of $t$), in
  $L_1(\R_+,dt)$ (see \cite[Theorem 3.6]{MR840845}). But this follows from
  the fact that, since $X \in L_p(\M{m}{\MM},\tau^{(m)})$ (resp. $X \in
  L_n(\M{m}{\MM},\tau^{(m)})$), the function $t \mapsto \mu_t(X)$ is in
  $L_p(\R_+,dt)$ (resp. $L_n(\R_+,dt)$.  This proves
  \eqref{eq=consequence_du_tcd_non_commutatif}.

  Now replace $z$ in \eqref{eq=consequence_du_tcd_non_commutatif} by the
  random variable $Z$ defined above, multiply by $\prod_{j=1}^n
  \overline{Z_j}^{\varepsilon_j}$ and take the expected value. Since $z$ is
  no longer fixed, $Y_r$ is denoted by $Y_r(z)$ to remember that $Y_r$
  depends on $z$. Since $Z$ only takes a finite number of values, equation
  \eqref{eq=consequence_du_tcd_non_commutatif} then becomes:
\[\frac{1}{r^n} \mathbb{E} \left[ \tau^{(m)}\left( \prod_{j=1}^n
  \overline{Z_j}^{\varepsilon_j} | S_{r Z}|^p - Q_n(\sum_{j=0}^n
 \binom{\nicefrac p 2}{j} \prod_{j=1}^n \overline{Z_j}^{\varepsilon_j}
 {Y_r(Z)}^j) \right)\right] \xrightarrow{r \to 0} 0.\]

Note that $Q_n(\sum_{j=0}^n \binom{\nicefrac p 2}{j} {Y_r(z)}^j)$ is, as a
function of $z = (z_1,\dots z_n)$, a polynomial in the $2n$ variables $z_i$
and $\overline{z_j}$ with coefficients in $L_1(\Mm\otimes\MM,\tau^{(m)})$
(this follows from H\"older's inequality and from the fact that $x_j \in
L_{n}(\MM,\tau)$). Moreover, if $P(z_1,\dots z_n)$ is such a polynomial,
\ie{}  $P(z)=\sum_{k,l \in \N^n, |k|+|l|\leq n} X_{k,l} z^k
\overline{z}^l$, then
\[\mathbb{E} \left[\prod_{j=1}^n \overline{Z_j}^{\varepsilon_j} P(Z)\right]
= X_{k^0,l^0},
\]
where $k^0 \in \N^n$ and $l^0\in \N^n$ are again defined by $k^0_j = 1$ if
$\varepsilon_j=1$, $k^0_j = 0$ else, and $l^0_j = 1 - k^0_j$. If
$D_{k^0,l^0}$ denotes the coefficient in front of $z^{k^0}
\overline{z}^{l^0}$ in $\sum_{j=0}^n \binom{\nicefrac p 2}{j} {Y_1(z)}^j$,
then one has:
\begin{eqnarray*}\mathbb{E} \left[ Q_n(\sum_{j=0}^n \binom{\nicefrac p
      2}{j}
\prod_{j=1}^n \overline{Z_j}^{\varepsilon_j} {Y_r(Z)}^j)\right] & = &
\mathbb{E} \left[Q_n(\sum_{j=0}^n \binom{\nicefrac p 2}{j}\prod_{j=1}^n
\overline{Z_j}^{\varepsilon_j} Y_1(rZ) )\right] \\ &=& r^n D_{k^0,l^0}. 
\end{eqnarray*}

Taking the trace $\tau^{(m)}$, dividing by $r^n$ and taking the limit as
 $r\to 0$ in \eqref{eq=consequence_du_tcd_non_commutatif}, one gets
\[\frac{1}{r^n} \mathbb{E} \left[ \prod_{j=1}^n
  \overline{Z_j}^{\varepsilon_j} \| S_{r Z}\|_p^p\right]
  \xrightarrow{r \rightarrow 0} \tau^{(m)}(D_{k^0,l^0}).\]

This shows
\eqref{eq=egalite_entre_normeP_et_trace_avec_adjoint_non_borne} since from
Lemma \ref{thm=lemme_combinatoire},
\[\tau^{(m)}(D_{k^0,l^0}) = \tau(x_1^{\varepsilon_1}
  x_2^{\varepsilon_2}\dots x_n^{\varepsilon_n })
  \sum_{k=0}^\alpha (n-k) \binom{\nicefrac p
  2}{n-k}\binom{\alpha}{k}.\]
\end{proof}

\subsection{Boundedness on $E \cap L_2$ of isometries on $E \subset L_p$}
In this subsection and in the next one, we study how isometric
properties for one $p$-norm imply boundedness (and isometric)
properties for the $q$-norms for $q \neq p$.

Here we first show that a unital map which is isometric between
subspaces of \nc{} $L_p$-spaces for $1\leq p<\infty$ is also
isometric for the $2$-norm. This is a \nc{} analogue of
\cite[Proposition 1]{MR0169081}, where the author proves that a unital
isometry between subspaces of commutative probability $L_p$-spaces is
also an isometry for the $L_2$-norm, and our proof is inspired by
Forelli's proof. 

Then in Theorem \ref{thm=equivalence_entre_dans_L2n_et_image}, we will
prove that any unital and $2$-isometric map between subspaces of \nc{}
$L_p$-spaces for $0<p<\infty$, $p \notin 2\N$ is also isometric for the
$2n$-norm for any $n \in \N \cup \{\infty\}$.

\begin{thm}
\label{thm=isometrie_Lp_est_isom_sur_L2}
Let $(\MM,\tau)$ and $(\NN,\widetilde \tau)$ be as in the introduction,
and $1\leq p<\infty$.

Let $x \in L_p(\MM,\tau)$ and $y \in L_p(\NN,\widetilde \tau)$ such that
for any $z \in \C$,
\begin{equation}
\label{eq=hypothese_isom_sur_Lp_pour_avoir_L2}
 \|\un + z x\|_p = \|\un + z y\|_p.
\end{equation}
Then $\|x\|_2 < \infty$ if and only if $\|y\|_2 < \infty$, and
$\|x\|_2 = \|y\|_2$.
\end{thm}

The following Lemma will be used; its proof was communicated to me by
Pisier.
\begin{lemma}
\label{thm=ineg_a_quatre_termes}
Let $A$ be a bounded operator on a Hilbert space $H$, and $p \geq
1$. Then
\begin{equation}
\label{eq=ineg_a_quatre_termes}
 |1+A|^p + |1-A|^p + |1+A^*|^p + |1-A^*|^p \geq 4
\end{equation}
\end{lemma}
\begin{proof}
By the operator convexity of the function $t \rightarrow t^r$ for $1
\leq r \leq 2$ and by an induction argument, it is enough to prove
\eqref{eq=ineg_a_quatre_termes} for $p = 1$.  For convenience we
denote by $C = |\un+A| + |\un-A| + |\un+A^*| + |\un-A^*|$.

By \cite[Corollary 1.3.7]{MR2284176}, for any operator $B$ on $H$,
the following operator on $H \oplus H$ is positive:
\[ \begin{pmatrix} |B| &  B^* \\ B & |B^*| \end{pmatrix}.\]
Replacing $B$ respectively by $\un + A$, $\un - A$, $\un + A^*$ and
$\un - A^*$ and adding the four resulting positive operators, we get
that the following operator is also positive:
\[ \begin{pmatrix} C &  4 \\ 4 & C\end{pmatrix}.\]

It is classical that this implies that $C \geq 4$ (see for example
\cite[Theorem 1.3.3]{MR2284176}).
\end{proof}
\begin{rem} 
  The Lemma is stated for bounded operators, but by approximation it also
  applies to closed densely defined unbounded operators.

  The inequality \eqref{eq=ineg_a_quatre_termes} does not hold for
  $0<p < 1$ (take $A = \un$). But if one could find a finitely
  supported probability measure $\nu$ on $\C \setminus \{0\}$ such
  that
\begin{equation}
\label{eq=positivite_aplusquatre_termes}
 \int \left(|\un + z A|^p + |\un + z A^*|^p \right) d \nu(z) \geq 2,
\end{equation}
then one would be able to get the conclusion of Theorem
\ref{thm=isometrie_Lp_est_isom_sur_L2} also for the values of $p$ for
which \eqref{eq=positivite_aplusquatre_termes} holds.
\end{rem}

\begin{proof}[Proof of Theorem \ref{thm=isometrie_Lp_est_isom_sur_L2}]
If $\|x\|_2 = \|y\|_2 = \infty$, there is nothing to prove.

If $\|x\|_2 < \infty$ and $\|y\|_2 < \infty$, then the fact that
$\|x\|_2 = \|y\|_2$ follows from Lemma
\ref{thm=egalite_entre_normeP_et_trace_avec_adjoint_nb} with $n = 2$,
$m=1$ and $(\varepsilon_1,\varepsilon_2) = (*,1)$ (and hence $\alpha =
1$). Indeed, by the hypothesis
\eqref{eq=hypothese_isom_sur_Lp_pour_avoir_L2}, the left-hand side in
equation
\eqref{eq=egalite_entre_normeP_et_trace_avec_adjoint_non_borne} does
not change if one takes $x_1 = x_2 = x \in L_p(\MM)$ or $x_1 = x_2 = y
\in L_p(\NN)$. Therefore the right-hand sides are also equal:
\[ \tau(x^* x) \left( 2 \binom{p/2}{2} + \binom{p/2}{1}\right) 
 = \tau(y^* y) \left( 2 \binom{p/2}{2} + \binom{p/2}{1}\right).\]

This implies that $\|x\|_2^2 = \tau(x^* x) = \tau(y^* y) = \|y\|_2^2$ 
since $2 \binom{p/2}{2} + \binom{p/2}{1} = p^2/4 \neq 0$.

Hence we only have to prove that if $\|x\|_2 < \infty$, then
$\|y\|_2<\infty$.

Denote by $C(x)$ and $C(y)$ the following operators:
\begin{eqnarray*}
C(x)= \sum_{\omega \in \{1,i,-1,-i\}} \left(|1 + \omega x|^p + |1 +
\omega x^*|^p - 2\right)\\ 
C(y)= \sum_{\omega \in \{1,i,-1,-i\}}
\left(|1 + \omega y|^p + |1 + \omega y^*|^p -2 \right)
\end{eqnarray*}

By Lemma \ref{thm=ineg_a_quatre_termes}, $C(x)$ and $C(y)$ are
positive operators, and by Lemma
\ref{thm=convergence_en_mesure_cas_NC}, $C(rx)/r^2$
(resp. $C(ry)/r^2$) converges in measure to $p^2(x^*x + x x^*)$
(resp. $p^2(y^*y + y y^*)$) as $r \rightarrow 0$. Moreover, for any
$r$, the hypothesis \eqref{eq=hypothese_isom_sur_Lp_pour_avoir_L2} implies that
\[ \tau(C(rx)) = 2 \sum_{\omega \in \{1,i,-1,-i\}}
\|1 + \omega r x\|_p^p - 8 = \widetilde \tau(C(ry)).\]

By Fatou's Lemma (\cite[Theorem 3.5]{MR840845}), we thus have that
\[2 p^2\|y\|_2^2 = \widetilde \tau\left(p^2(y^*y + y y^*)\right) \leq 
\liminf_{r\rightarrow 0} \tau(C(rx)/r^2). \]

We now use Fack and Kosaki's dominated convergence theorem
\cite[Theorem 3.6]{MR840845} to prove that $\tau(C(rx))/r^2
\rightarrow 2 p^2 \|x\|_2^2$. By the domination Lemma
\ref{thm=domination_cas_NC} and the property
\eqref{eq=inegalite_sing_num_add} of singular numbers, there are
constants $C,K>0$ such that
\[\mu_t(C(rx)/r^2) \leq C(\mu_{t/K}(x)^2 + \mu_{t/K}(x)^p).\]
As in the proof of Lemma
\ref{thm=egalite_entre_normeP_et_trace_avec_adjoint_nb}, this is
enough to deduce that
\[ \lim_{r\rightarrow 0} \tau(C(rx)/r^2) = \tau\left(\lim_{r\rightarrow 0} 
C(rx)/r^2\right) = 2 p^2 \|x\|_2^2.\]

This concludes the proof.
\end{proof}

\subsection{Boundedness on $E \cap L_{2n}$ for all $n$ of $2$-isometries 
on $E \subset L_p$} Here we prove that a unital $2$-isometric map
between unital subspaces of \nc{} $L_p$-spaces maps a
bounded operator to a bounded operator. The general idea is to prove
by induction on $n$ that such a map is also an isometry for the
$q$-norm when $q=2n$ and to make $n$ grow to $\infty$. The idea of the
proof is similar to the proof of Theorem
\ref{thm=isometrie_Lp_est_isom_sur_L2}. The precise statement is:

\begin{thm}%
\label{thm=equivalence_entre_dans_L2n_et_image}
Let $0<p<\infty$, $p$ not an even integer. Let $x \in L_p(\MM,\tau)$ and $y
\in L_p(\NN, \widetilde \tau)$ such that, for any $a\in \M{2}{\C}$
\[\left \| \un + a \otimes x \right\|_p = \left \| \un + a
  \otimes y \right\|_p.\]
Then for any $n \in \N^* \cup \{\infty\}$, $x \in L_{2n}(\MM)$ if and
only if $y \in L_{2n}(\NN)$, and when this holds $\|x\|_{2n} =
\|y\|_{2n}$.
\end{thm}%

The Theorem is proved with the use of Fatou's Lemma and expansions in
power series of operators of the form $|\un + a|^p$ for $a$ satisfying
$a^2=0$. More precisely, for such an $a$, we derive an expression of the
following form (Corollary \ref{thm=expression_de_p-puiss_avec_psi} and
Lemma \ref{thm=proprietes_de_psi}):
\[\left|1+a\right|^p + \left|1-a\right|^p + \left|1+a^*\right|^p +
\left|1-a^*\right|^p \simeq \sum_{n=0}^N \lambda_n |a|^{2n} + \lambda_n
|a^*|^{2n}\]
and are able to use a qualitative study of differential equations
(Lemma \ref{thm=equa_diff_qualitatif}) to prove the positivity (or
negativity) of the difference of the two above terms.

\begin{lemma}
\label{thm=lemme_somme_de_puiss_a_i}
Let $a$ be an element of a $*$-algebra such that $a^2=0$.
Then if one denotes by $a_1$, $a_2$, $a_3$, $a_4$ the expressions
\begin{eqnarray*}
a_1 = a^* a + a + a^*\\
a_2 = a^* a - a - a^*\\
a_3 = a a^* + a + a^*\\
a_4 = a a^* - a - a^*
\end{eqnarray*}
then for any integer $m \geq 1$,
\begin{equation*}\sum_{j=1}^4 a_j^m = 2 P_m(a^*a) + 2 P_m(a a^*),
\end{equation*}
where the polynomial $P_m$ is defined by
\begin{equation}
\label{eq=def_de_Pm}
P_m(X) = \left(\frac{X + \sqrt{X^2+4X}}{2}\right)^m +
\left(\frac{X - \sqrt{X^2+4X}}{2}\right)^m \textrm{  for any } X \in \R.
\end{equation}
\end{lemma}
\begin{proof}
  We can assume that $a$ is free element satisfying $a^2 = 0$, so that
  there are well defined polynomials $A_m$, $B_m$, $C_m$ and $D_m$ in
  $\R[X]$ such that
\[(a_1)^m = A_m(a^*a) + a B_m (a^*a) + C_m(a^*a) a^* + a D_m(a^*a) a^*.\]
Thus we can write
\[\sum_{j=1}^4 a_j^m = 2 P_m(a^*a) + 2 P_m(a a^*)\]
with $P_m = A_m + X D_m \in \R[X]$.

It is easy to check that the sequences of polynomials $(A_m)_m$ and
$(D_m)_m$ (and hence $(P_m)$) satisfy the following induction relations:
\begin{eqnarray*}
A_{m+2}(X) = X\left(A_{m+1}(X)+A_{m}(X)\right) & \textrm{if $m \geq 0$}\\
D_{m+2}(X) = X\left(D_{m+1}(X)+D_{m}(X)\right)  & \textrm{if $m \geq 1$}\\
P_{m+2}(X) = X\left(P_{m+1}(X)+P_{m}(X)\right)  & \textrm{if $m \geq 1$}
\end{eqnarray*}

But the right-hand side of \eqref{eq=def_de_Pm} also satisfies the same
relation, it is therefore enough (and trivial) to check that equality
\eqref{eq=def_de_Pm} holds for $m=1$ and $m=2$.
\end{proof}
\begin{lemma}
  Let $a$ be a closed densely defined operator affiliated to a von Neumann
  algebra $(\MM,\tau)$ such that $a^2=0$. Let $a_i, i=1, 2, 3, 4$ be as in
  Lemma \ref{thm=lemme_somme_de_puiss_a_i}.  Then for any continuous
  function $f:[-1,\infty) \rightarrow \R$,
\begin{multline}
\label{eq=valeur_somme_fai}
\sum_{j=1}^4 f(a_i)^m =
2 f\left(\frac{a^* a + \sqrt{(a^*a)^2+4a^*a}}{2}\right) +
2 f\left(\frac{a^* a - \sqrt{(a^*a)^2+4a^*a}}{2}\right) +\\
2 f\left(\frac{a a^* + \sqrt{(a a^*)^2+4a a^*}}{2}\right) +
2 f\left(\frac{a a^* - \sqrt{(a a^*)^2+4a a^*}}{2}\right)  - 2f(0).
\end{multline}
\end{lemma}
\begin{proof}
Lemma \ref{thm=lemme_somme_de_puiss_a_i} implies that
\eqref{eq=valeur_somme_fai} holds when $f$ is a polynomial. By continuity
of the continuous functional calculus (with respect to the measure topology
when $a$ is unbounded), \eqref{eq=valeur_somme_fai} thus holds for any
continuous $f$.
\end{proof}%
\begin{corollaire}%
\label{thm=expression_de_p-puiss_avec_psi}
Let $0<p<1$ and $a$, as above, satisfying $a^2=0$. Then
\[ \left|1+a\right|^p + \left|1-a\right|^p + \left|1+a^*\right|^p +
\left|1-a^*\right|^p = 2 \psi(a^*a) + 2 \psi(a a^*) - 4,\]
where $\psi$ is the function $\psi:\R^+ \rightarrow \R$ defined by
\[\psi(t) =  \left(1+\frac{t + \sqrt{t^2+4t}}{2}\right)^{\nicefrac p 2} +
 \left(1+\frac{t - \sqrt{t^2+4t}}{2}\right)^{\nicefrac p 2}.\]
\end{corollaire}
\begin{proof}
This is immediate since with the notation above, $\left|1+a\right|^p
= (1+a_1)^{p/2}$, $\left|1-a\right|^p = (1+a_2)^{p/2}$,
$\left|1+a^*\right|^p = (1+a_3)^{p/2}$ and $\left|1-a^*\right|^p =
(1+a_4)^{p/2}$
\end{proof}
Let us study the function $\psi$. 

\begin{proposition}
\label{thm=proprietes_de_psi}
The following properties hold for $\psi$:
\begin{enumerate}
\item \label{item=equa_diff_psi}
$\psi$ is a solution to the following
differential equation on $\R^+$:
\begin{equation}
\label{eq=equa_diff_de_psi}
(t^2+4t)y'' + (t+2) y' - \frac{p^2}{4} y = 0. 
\end{equation}

\item \label{item=dse_psi} $\psi$ has an expansion in power series
around $0$, more precisely for $|t|<4$,
\begin{equation}
\label{eq=dse_psi}
\psi(t) = \sum_{n \geq 0} \frac{2}{(2n)!} \prod_{k=0}^{n-1} 
\left(\frac{p^2}{4} - k^2\right) t^n = \sum_{n \geq 0} \lambda_n t^n.
\end{equation}

\item \label{item=positivite_psi}
 For any $0<p<\infty$, for any $t \in \R^+$ and any $N\in \N$,
\begin{equation}
\label{eq=positivite_psi}
\psi(t) - \sum_{n=0}^N \lambda_n t^n \left\{
\begin{array}{ll} 
\geq 0 & \textrm{if } p \geq 2N \textrm{ or } 
\lfloor N -\frac p 2 \rfloor \textrm{ is odd.}\\
\leq 0 & \textrm{else.}
\end{array}
\right.
\end{equation}
\end{enumerate}
\end{proposition}
In this proposition, for a real number $t$, the symbol $\lfloor t
\rfloor$ denotes the largest integer smaller than or equal to $t$.
\begin{proof}
  Checking (\ref{item=equa_diff_psi}) is just an easy computation, the
  details are left to the reader. It is also easy to see that $\psi$ has an
  expansion in power series around $0$, and (\ref{item=dse_psi}) follows
  from the fact that both left-hand and right-hand sides of
  \eqref{eq=dse_psi} satisfy \eqref{eq=equa_diff_de_psi} and have value $2$
  at $t=0$.

Let us prove (\ref{item=positivite_psi}). Let us fix $p$ and $N$. As a
function of $t$, the left-hand side of \eqref{eq=positivite_psi} satisfies
the following differential equation:
\[(t^2+4t)y'' + (t+2) y' - \frac{p^2}{4} y = (2N+1)(2N+2) \lambda_{N+1} t^N.\]

Moreover, (\ref{item=dse_psi}) shows that the left-hand side of
\eqref{eq=positivite_psi} and its derivative has the same sign as
$\lambda_{N+1}$ when $t$ is small (with $t>0$).

Note also that $\lambda_{N+1} \geq 0$ if $p \geq 2 N$ or if $\lfloor N
-\frac p 2 \rfloor$ is odd, that $\lambda_{N+1} \leq 0$ else.

The fact (\ref{item=positivite_psi}) thus follows from Lemma
\ref{thm=equa_diff_qualitatif} applied to $\pm\left(\psi(t) -
\sum_{n=0}^N \lambda_n t^n\right)$ depending on the sign of
$\lambda_{N+1}$.
\end{proof}
\begin{lemma}
\label{thm=equa_diff_qualitatif}
Let $a$, $b$, $c$ and $d$ be continuous functions on $\R^+$
such that for any $t>0$,
\begin{eqnarray*}
a(t)>0\\
c(t)<0\\
d(t)>0
\end{eqnarray*} 
Let $y$ be a $C^2$ function on $\R^+$ solution of $a y'' + b
y' + cy = d$, and $t_0>0$ such that $y(t_0)>0$ and
$y'(t_0)>0$. Then $y(t)>0$ for any $t \geq t_0$.
\end{lemma}
\begin{proof}
We prove that $y'(t) > 0$ for any $t \geq t_0$. Assume that it
is not true, and take $t_1 = \min\{t>t_0, y'(t)=0\}$. Since
$y'(t_1)=0$ and $y'(t) > 0$ if $t_0< t < t_1$, we have that
$y''(t_1) \leq 0$.

On the other hand, since $y' \geq 0$ on $(t_0,t_1)$, $y(t_1)\geq y(t_0)>0$.
Thus $y''(t_1) = (d(t_1) - c(t_1) y(t_1))/a(t_1)>0$, which is a
contradiction.
\end{proof}

It is now possible to derive the main result of this part:
\begin{lemma}
\label{thm=lien_entre_etre_dans_L2n_et_norme_p}
Let $0<p<\infty$ and $(\lambda_n)_{n \geq 0} \in \R^\N$ defined by
\eqref{eq=dse_psi}. Take $a \in L_p(\MM,\tau)$ such that $a^2=0$, and fix an
integer $N>0$ with $\lambda_{N}\neq 0$.

\begin{enumerate}
\item If $\|a\|_{2N-2}<\infty$.
\begin{equation}
\label{eq=pour_prouver_a_dans_L_2n}
 \|a\|_{2N}^{2N} \leq \liminf_{t\rightarrow 0} \frac{1}{2
    \lambda_{N} t^{2N}} \left(\|\un + ta\|_p^p + \|\un - ta\|_p^p - 2 - 2
    \sum_{n=1}^{N-1} \lambda_n t^{2n}\|a\|_{2n}^{2n}\right)
\end{equation}

\item 
Moreover,if $\|a\|_{2N} <\infty$, the previous inequality becomes an
equality. More precisely, 
\begin{equation}
\label{eq=pour_calculer_norme_2n_de_a}
  \|a\|_{2N}^{2N} =  \lim_{t\rightarrow 0} \frac{1}{2
    \lambda_{N} t^{2N}} \left(\|\un + ta\|_p^p + \|\un - ta\|_p^p - 2 - 2
    \sum_{n=1}^{N-1} \lambda_n t^{2n}\|a\|_{2n}^{2n}\right)
\end{equation}
\end{enumerate}
\end{lemma}
\begin{proof}
The first fact is a consequence of the properties of $\psi$ and of Fatou's
lemma.

Denote by $b(t,a)$ the following (unbounded) operator affiliated with $\MM$:
\[b(t,a) = \frac{1}{\lambda_N t^{2N}}\left(\psi(t^2 a^* a) -
  \sum_{n=0}^{N-1} \lambda_n t^{2n} (a^* a)^n \right).\] In this equation,
$(a^* a)^0$ is equal to $\un_\MM$. Note that the operators $b(t,a)$ are
affiliated with the commutative von Neumann algebra generated by $a^*a$,
which is isomorphic to the space of (classes of) bounded measurable
functions on some probability space $(\Omega, \mu)$.

Then \eqref{eq=dse_psi} implies that $b(t,a) \rightarrow (a^*a)^N$ in the
measure topology as $t \rightarrow \infty$ (in fact the convergence holds
almost surely if the operators are viewed as functions on $\Omega$). But
\eqref{eq=positivite_psi} also implies that $b(t,a) \geq 0$.  Thus one can
apply Fatou's lemma to conclude that
\begin{equation}
\label{eq=consequence_de_fatou}
\|a\|_{2N}^{2N} = \tau((a^*a)^N) \leq \liminf_{t \rightarrow 0}
\tau(b(t,a)).
\end{equation}

Replace $a$ by $a^*$ in the preceding inequality, and add the two equations
to get (using $\|a^*\|_q = \|a\|_q$ for any real $q$)
\[ 2\|a\|_{2N}^{2N} \leq \liminf_{t \rightarrow 0} \frac{1}{\lambda_{N}
  t^{2N}} \tau \left( \psi(t^2 a^* a) + \psi(t^2 a a^*) - \sum_{n=0}^{N-1}
  \lambda_n t^{2n} ((a^* a)^n + (a a^*)^n) \right).\]

Applying Corollary \ref{thm=expression_de_p-puiss_avec_psi} and the
linearity of the trace yields to the desired conclusion (since $(a a^*)^n$
and $(a^* a)^n$ belong to $L_1(\MM)$ for $n \leq N-1$).

To prove the second fact, we prove that if $\|a\|_{2N} < \infty$, then
equality holds in \eqref{eq=consequence_de_fatou}. But this follows from
the (classical) dominated convergence theorem since $\left|\psi(t) -
  \sum_{n=0}^N \lambda_n t^n \right| \leq C (t^N + t^{p/2})$ for some
constant $C$ not depending on $t \in \R$.
\end{proof}

The proof of Theorem \ref{thm=equivalence_entre_dans_L2n_et_image} follows:

\begin{proof}[Proof of Theorem
  \ref{thm=equivalence_entre_dans_L2n_et_image}]
First note that the statement for $n= \infty$ follows from the one for
$n \in \N$, since $\|x\|_\infty = \lim_{n \rightarrow \infty}
\|x\|_{2n}$. So we focus on the case when $n$ is a positive integer.

  The idea is to construct operators related to $x$ and $y$ of zero
  square by putting then in a corner of a $2$ by $2$ matrix, and then
  to use Lemma \ref{thm=lien_entre_etre_dans_L2n_et_norme_p}. So let
  us denote $a(x)$ and $a(y)$ the operators
\begin{eqnarray*}
a(x) = \begin{pmatrix}
0 & x\\0 & 0 \end{pmatrix} \in \M{2}{L_p(\MM)} \simeq L_p(\M 2 \MM)\\
a(y) = \begin{pmatrix}
0 & y\\0 & 0 \end{pmatrix} \in \M{2}{L_p(\MM)} \simeq L_p(\M 2 \MM)
\end{eqnarray*}
Note that $a(x)^2=0$ that for any $q\in \R \cup \{\infty\}$, $\|a(x)\|_q =
2^{- \nicefrac 1 q} \|x\|_q$, and that the same holds for $y$. Moreover
$\|\un +t a(x)\|_p = \|\un + t a(y)\|_p$ for any $t\in \R$. It is thus
enough to prove that if $\|a(x)\|_{2n}<\infty$, then $\|a(y)\|_{2n}<\infty$
and $\|a(y)\|_{2n} = \|a(x)\|_{2n}$. We prove this by induction on $n$.

So take $N>0$, assume that the aforementioned statement holds for any
$n<N$. (note that we assume nothing if $N=1$). 
Suppose that $\|a(x)\|_{2N}<\infty$. Then by induction hypothesis for any
$n<N$, $\|y\|_{2n} = \|x\|_{2n}$. Thus the right-hand side of
\eqref{eq=pour_prouver_a_dans_L_2n} is the same when $a$ is replaced by
$a(y)$ or by $a(x)$. But for $a=a(x)$, it is equal, by
\eqref{eq=pour_calculer_norme_2n_de_a}, to $\|a(x)\|_{2N}$. Hence
\eqref{eq=pour_prouver_a_dans_L_2n} proves that $\|a(y)\|_{2N} \leq
\|a(x)\|_{2N} < \infty$.

Applying \eqref{eq=pour_calculer_norme_2n_de_a} again with $a(y)$ yields to
$\|a(y)\|_{2N}=\|a(x)\|_{2N}$.
\end{proof}

\section{Proof of Theorem \ref{thm=thm_principal}}
\label{section=preuve_resultat_principal}
In this section we develop some consequences of Lemma
\ref{thm=egalite_entre_normeP_et_trace_avec_adjoint}. We are given
$(\MM,\tau)$ and $(\NN, \widetilde \tau)$ two von Neumann algebras with
normal faithful tracial states.

Let $x_1, \dots x_n \in \MM$ and $y_1, \dots y_n \in \NN$. The
\nc{} analogue (in the bounded case) of Theorem
\ref{thm=theoreme_1_rudin} is:
\begin{thm}
\label{thm=equirepartition}
Let $0< p <\infty$ such that $p \neq 2,4,6 \dots$ is not an even integer.
Suppose that for all $m \in \N$ and all $a_1 \dots a_n \in \Mm$,
\[\|\un + \sum a_i \otimes x_i\|_p = \|\un + \sum a_i \otimes y_i\|_p.\]

Then the $n$-uples $(x_1, \dots x_n)$ and $(y_1,\dots y_n)$ have the
same $*$-distributions. More precisely, for all $P \in \C\left\langle
X_1, \dots X_{2n} \right\rangle$ polynomial in $2n$ non commuting
variables,
\begin{equation}
\label{eq=egalite_des_distributions}
\tau\paren{P(x_1,\dots x_n,x_1^*,\dots x_n^*)} = \widetilde
\tau\paren{P(y_1,\dots y_n,y_1^*,\dots y_n^*)}.
\end{equation}
\end{thm}
This theorem relies on Lemma
\ref{thm=egalite_entre_normeP_et_trace_avec_adjoint} and on the following
Lemma:
\begin{lemma}
\label{thm=lemme_coeff_binomial_non_nul}
Let $N,\alpha \in \N$ be integers such that $N\geq 1$ and $\alpha \leq
N/2$. Then if $p$ is a positive number such that $p \notin 2\N$ or $p \geq 2(N-
\alpha)$, then 
\[\sum_{k=0}^\alpha (N-k) \binom{\nicefrac p 2}{N-k}\binom{\alpha}{k} \neq 0.\]
\end{lemma}
\begin{proof}
Take $\alpha, N $ and $p$ as in the Lemma. Since $(N-k) \binom{\nicefrac p
2}{N-k} = {\nicefrac p 2} \binom{\nicefrac p 2 -1}{N-k-1}$, showing the
Lemma is the same as showing that
\begin{equation}
\label{eq=truc_a_montrer_non_nul}
\sum_{k=0}^\alpha \binom{\nicefrac p 2-1}{N-k -1 }\binom{\alpha}{k}
\neq 0.
\end{equation}

For every real number $\beta$, let us consider the left-hand side of
\eqref{eq=truc_a_montrer_non_nul} where $\nicefrac p 2-1$ is replaced by
$\beta$. Since $\binom{\beta}{n}$ is a polynomial function in $\beta$ of
degree $n$, the expression $P(\beta) \egdef \sum_{k=0}^\alpha \binom{\beta}{N-k -1
}\binom{\alpha}{k}$ is a polynomial in $\beta$ of degree $N-1$. To prove
that it takes nonzero values for $\beta = p/2 -1$, we show that it has
$N-1$ roots different from $p/2 -1$. More precisely, we show that if
$\beta$ is an integer such that $-\alpha \leq \beta \leq N-\alpha -2$, then
$P(\beta)=0$. 

First if $\beta$ is an integer between $0$ and $N-\alpha -2$ included, then
for any $0 \leq k \leq \alpha $, it is immediate to check from the
definition \eqref{eq=def_coeff_binomial} that $\binom{\beta}{N-k-1}=0$,
which implies $P(\beta)=0$.

The second fact to check is that if $l$ is an integer such that $1 \leq l
\leq \alpha$, then $P(-l)=0$. Let us fix such an $l$. Then writing
$\binom{-l}{N-k-1} = (-1)^{N-k-1}\binom{N-k+l-2}{l-1}$ we get
\[ P(-l) = \sum_{k=0}^\alpha \binom{-l}{N-k-1}\binom{\alpha}{k} =
\sum_{k=0}^\alpha\binom{\alpha}{k} (-1)^{N-k-1}\binom{N-k+l-2}{l-1}.\]

  It only remains to note that $l$ and $N$ being fixed,
  $\binom{N-k+l-2}{l-1}$ is (as a function of $k$) a polynomial of degree
  $l-1 < \alpha$. The equality $P(-l)=0$ arises from the fact that if $1
  \leq i < \alpha$,
\[ \sum_{k=0}^\alpha \binom{\alpha}{k} (-1)^k k^i = 0.\]
\end{proof}
Theorem \ref{thm=equirepartition} follows:
\begin{proof}[Proof of Theorem \ref{thm=equirepartition}]
By linearity it is enough to prove
\eqref{eq=egalite_des_distributions} when $P$ is a monomial. The fact
to be proved is that for every finite sequence $i_1, \dots i_N$ of
indices between $1$ and $n$, and for every sequence $\varepsilon_1,
\dots \varepsilon_N \in \left\{1,*\right\}$,
\[\tau\paren{\prod_k x_{i_k}^{\varepsilon_k}} = 
\widetilde \tau \paren{\prod_k y_{i_k}^{\varepsilon_k}}.\]

But from lemma \ref{thm=egalite_entre_normeP_et_trace_avec_adjoint}, if
$\alpha$ is the number of indices $k$ such that $\varepsilon_k = *$ and
$\varepsilon_{k +1 \mod N} = 1$, we have
\[\tau\paren{\prod_k x_{i_k}^{\varepsilon_k}} \sum_{k=0}^\alpha (N-k)
\binom{\nicefrac p 2}{N-k}\binom{\alpha}{k} = \widetilde \tau
\paren{\prod_k y_{i_k}^{\varepsilon_k}} \sum_{k=0}^\alpha (N-k)
\binom{\nicefrac p 2}{N-k}\binom{\alpha}{k} .\]

This implies that $\tau\paren{\prod_k x_{i_k}^{\varepsilon_k}} = \widetilde
\tau \paren{\prod_k y_{i_k}^{\varepsilon_k}}$ since from Lemma
\ref{thm=lemme_coeff_binomial_non_nul} $\sum_{k=0}^\alpha (N-k)
\binom{\nicefrac p 2}{N-k}\binom{\alpha}{k} \neq 0$ if $p \notin 2\N$.
\end{proof}

Theorem \ref{thm=thm_principal} is an immediate consequence of Theorem
\ref{thm=equirepartition}, Theorem
\ref{thm=equivalence_entre_dans_L2n_et_image} and of the following
well-known lemma:
\begin{lemma}
\label{thm=lemme_equidistribution_implique_isomorphisme}
Let $(\MM,\tau)$ and $(\NN,\widetilde \tau)$ be two von Neumann
algebras equipped with faithful normal tracial states, and let
$(x_i)_{i\in I} \in \MM$ and $(y_i)_{i\in I} \in \NN$ be \nc{} random
variables that have the same $*$-distribution. Then the von Neumann
algebras generated respectively by the $x_i$'s and the $y_i$'s are
isomorphic, with a normal isomorphism sending $x_i$ on $y_i$ and
preserving the trace.
\end{lemma}
\begin{proof}[Proof of Theorem \ref{thm=thm_principal}]
Let $(x_i)_{i \in I}$ be a family spanning $E$. If $y_i = u(x_i)$ for
any $i \in I$, then Theorem
\ref{thm=equivalence_entre_dans_L2n_et_image} shows that
$\|y_i\|_\infty<\infty$, which is equivalent to the fact that $y_i \in
\NN$. By Theorem \ref{thm=equirepartition}, the families
$(x_i,x_i^*)$ and $(y_i,y_i^*)$ have the same distribution
and so by Lemma
\ref{thm=lemme_equidistribution_implique_isomorphisme}, $u$ extends to
a trace preserving isomorphism between the von Neumann algebras
generated by the $x_i$'s and $y_i$'s respectively.
\end{proof}

It is also possible to get some approximation results using
ultraproducts:

\subsection{Approximation results}

\begin{corollaire}
Let $(\MM_\alpha,\tau_\alpha)_{\alpha \in A}$ be a net of von Neumann
algebras equipped with normal faithful normalized traces. Let $I$ be a
set, and for all $\alpha$, let $(x_i^\alpha)_{i\in I} \in
{\MM_\alpha}^I$ such that for all $i \in I$, the net
$(x_i^\alpha)_{\alpha}$ is uniformly bounded, \ie{}
$\sup_\alpha \|x_i^\alpha\| < \infty$. Assume that there is a family
$(y_i)_{i \in I}$ in a von Neumann algebra $(\NN,\widetilde \tau)$ and
a $p \notin 2 \N$ such that for all integer $n$ and all finitely
supported family $(a_i)_{i \in I} \in M_n$, the following holds:
\begin{equation}
\label{eq=limite_des_normes_p}
 \lim_{\alpha} \|1 + \sum_i a_i \otimes x_i^\alpha\|_p =
 \|1 + \sum_i a_i \otimes y_i\|_p. 
\end{equation}

Then the net $\left((x_i^\alpha)_{i} \right)_\alpha$ converges in
$*$-distribution to $(y_i)_i$. Moreover \eqref{eq=limite_des_normes_p}
holds with $p$ replaced by any $0<q<\infty$.
\end{corollaire}

\begin{proof}
Indeed let $\mathcal U$ be any ultraproduct on $A$ finer that the net
$(\alpha)$, and for $i \in I$ consider $x_i$ the image of
$(x_i^\alpha)_{\alpha \in A}$ in the von Neumann ultraproduct $\MM =
\prod_{\mathcal U} \MM_\alpha$. If $\MM$ is equipped with the tracial
state $\tau = \lim_{\mathcal U} \tau_\alpha$, then the assumption
\eqref{eq=limite_des_normes_p} implies that for all $m$ and all $a_i
\in M_m$,
\[\|1 + \sum_i a_i \otimes x_i\|_p = \|1 + \sum_i a_i \otimes y_i\|_p.\]

Lemma \ref{thm=egalite_entre_normeP_et_trace_avec_adjoint} implies
that $(x_i)_i$ and $(y_i)_i$ have the same $*$-distribution. This
exactly means that $(x_i^\alpha)_{i}$ converges in $*$-distribution to
$(y_i)_i$ as $\alpha \in \mathcal U$.

Since this holds for any ultrafilter $\mathcal U$ finer than the net
$(\alpha)$, this proves the convergence in $*$-distribution of the net
$\left((x_i^\alpha)_{i} \right)_\alpha$ to $(y_i)_i$.  The fact that
\eqref{eq=limite_des_normes_p} then holds with $p$ replaced by any
$0<q<\infty$ is immediate.
\end{proof}

Theorem \ref{thm=thm_principal} can also be reformulated in the
operator space setting:

\subsection{Reformulation in the operator space setting}
\label{part=espace_doperateurs}
Let $\mathcal M \subset \B{H}$ be a von Neumann algebra equipped with
a normal faithful trace $\tau$ satisfying $\tau(\un)=1$. Let $E$ be a
linear subspace of $\mathcal M$. There are several ``natural'' operator space
structures on $E$:

For all $1\leq p \leq \infty$, the \nc{} $L_p$-spaces $L_p(\mathcal
M,\tau)$ are equipped with a natural operator space structure (see
\cite[chapter 7]{MR2006539}). (when $p = \infty$, $L_p(\mathcal M,\tau)$ is
the von Neumann algebra $\mathcal M$ with its obvious operator space
structure).
%% for $p = \infty$, $L_p(\mathcal M,\tau)$ is simply equal to $\mathcal M$
%% which is equipped with the operator space structure coming from the
%% embedding $\mathcal M \subset \B{H}$. For $p = 1$, $L_1(\mathcal M,\tau)$ is the predual of
%% $\mathcal M$, and is equipped with the \emph{opposite} operator space structure
%% coming from the operator space of the dual $\mathcal M^*$ of $\mathcal M$. Then
%% $(L_1(\mathcal M,\tau), L_\infty(\mathcal M,\tau))$ is a compatible couple of operator
%% space, and the identification (at Banach space level) of $L_p(\mathcal M,\tau)$
%% with the interpolated space $(L_\infty(\mathcal M,\tau),
%% L_1(\mathcal M,\tau))_{\nicefrac 1 p}$ allows to define on $L_p(\mathcal M,\tau)$ a
%% natural operator space structure (the interpolated operator space
%% structure).

Then the linear embedding $E \subset L_p(\mathcal M,\tau)$ allows to
define, for all $1\leq p \leq \infty$, an operator space structure on
$E$, which we denote by $O_p(E)$.

In this setting, Theorem \ref{thm=thm_principal} states that if $E$ is
a linear subspace of $\mathcal M$ containing the unit and if $1\leq
p<\infty$ and $p \notin 2 \N$, then the operator space structure
$O_p(E)$ together with the unit entirely determines the von Neumann
algebra generated by $E$ and the trace on it. In particular it
determines all of the other operator space structures $O_q(E)$ for all
$1\leq q \leq \infty$.

More precisely:
\begin{corollaire}
\label{thm=deuxieme_corollaire}
Let $\un_\mathcal M \in E \subset \mathcal M$ be as above, $(\mathcal
N,\widetilde \tau)$ be another von Neumann algebra equipped with a
normal faithful tracial state, $u:E \rightarrow \mathcal N$ be a
unit preserving linear map and $1 \leq p<\infty$ with $p \notin 2 \N$.

If $u:O_p(E) \rightarrow L_p(\mathcal N,\widetilde \tau)$ is a
complete isometry, then $u$ uniquely extends to an isomorphism between
the von Neumann subalgebras generated by $E$ and its image; moreover
$u$ is then trace preserving. In particular, for all $1\leq q \leq
\infty$, $u: O_q(E) \rightarrow L_q(\mathcal N,\widetilde \tau)$ is a
complete isometry.
\end{corollaire}

\begin{proof}
The proof is a reformulation of Theorem \ref{thm=thm_principal} once we
know the two following results from the theory of \nc{}
vector valued $L_p$-spaces developed in \cite{MR1648908}:

A map $u:X->Y$ between two operator spaces is completely isometric if
and only if for all $n$, the map $u \otimes id:S_p^n(X) \rightarrow
S_p^n(Y)$ is an isometry (Lemma 1.7 in \cite{MR1648908}). More
precisely, for any $n \in \N$ and any $1 \leq p \leq \infty$,
\[\left\|u \otimes id:S_p^n(X) \rightarrow S_p^n(Y)\right\| =
   \left\|u \otimes id: \M{n}{X} \rightarrow \M n Y\right\|\]
%                  page 23 
The second result is Fubini's theorem, which states that
$S_p^n\bigl(L_p(\mathcal M,\tau)\bigr) \simeq L_p \bigl(\Mn \otimes A,
\tr_n \otimes \tau \bigr)$ isometrically (and even completely
isometrically, but this is of no use here). See Theorem 1.9 in
\cite{MR1648908}.
%    page 25

These two results together prove that the hypotheses in Corollary
\ref{thm=deuxieme_corollaire} imply those in Theorem
\ref{thm=thm_principal}, and thus the result is proved.
\end{proof}

\subsection{On the necessity of taking matrices of arbitrary size}
Here we discuss the necessity of taking matrices of arbitrary size in
Theorem \ref{thm=thm_principal}. In view of Theorem
\ref{thm=thm_principal} a natural question is thus:

Let $p \in \R$. Consider the class $\mathcal E_{p,1}$ of all linear
maps $u$ between subspaces of \nc{} $L_p$ spaces constructed on von
Neumann algebras equipped with a \nff{} normalized trace. Is there an
integer $n$ such that for any such $u: E \rightarrow F$, if
\eqref{eq=hypothese_p_c.isometrie} holds for all $x \in \M{n}{E}$, then 
it holds for any $m$ and any $x \in \M{m}{E}$? The smallest such
integer will be denoted by $n_{p,1}$.

A similar question is:

Let $p \in \R$. Consider the class $\mathcal E_{p}$ of all linear maps
$u$ between subspaces of \nc{} $L_p$ spaces constructed on von Neumann
algebras equipped with a normal semifinite faithful normalized
trace. Is there an integer $n$ such that for any such $u \in \mathcal
E_p$, if $u$ is $n$-isometric, then $u$ is completely isometric? The
smallest such integer will be denoted by $n_p$.

As was noted in the introduction, the transposition map from $\Mn$ to
$\Mn$ ($n \geq 2$) shows that, except for $p=2$, we necessarily have
$n_{p,1}>1$ and $n_{p} >1$.

When $p \notin 2 \N$, it is not clear whether $n_{p,1}<\infty$ (or
$n_p <\infty$).

In the opposite direction, as announced in the introduction, when $p = 2m
\in 2 \N$, then it is not hard to prove that $n_{p,1} \leq m$ and $n_p \leq
m$.

\begin{thm}
\label{thm=p=2m_m_isometrie_implique_c.isom}
Let $p = 2 m \in 2\N$. Let $(\MM,\tau)$, $(\NN,\widetilde \tau)$ be as
in Theorem \ref{thm=thm_principal}.

Let $E \subset L_p(\MM,\tau)$ be a subspace and $u:E \rightarrow
L_p(\NN,\widetilde\tau)$ be a linear map.

Assume that for all $x \in \M{m}{E}$, the following
equality between the $p$-norms holds:
\begin{equation}
\label{eq=egalite_des_normes_p_ppair}
\|\forall x \in \M{m}{E}, \ \ \ \un_m \otimes \un_\MM + x\|_{2m} = \|\un_m \otimes \un_\NN + (\id\otimes
u)(x)\|_{2m}.
\end{equation}

Then in fact this equality holds for $x \in \M n E$ for every $n \in \N$:
\[\|\un_n \otimes \un_\MM + x\|_{2m} = \|\un_n \otimes \un_\NN + (\id\otimes
u)(x)\|_{2m}.\]
\end{thm}

\begin{thm}
\label{thm=p=2m_m_isometrie_implique_c.isom_semifini}
Let $p = 2 m \in 2\N$. Let $(\MM,\tau)$, $(\NN,\widetilde \tau)$ be
(exceptionally) von Neumann algebras with normal faithful
\emph{semifinite} traces.

Let $E \subset L_p(\MM,\tau)$ be a subspace and $u:E \rightarrow
L_p(\NN,\widetilde\tau)$ be a linear map.

If $u$ is $m$-isometric (\ie{} $\|x\|_{L_p(\tau^{(m)})} =
\|(\id\otimes u)(x)\|_{L_p(\widetilde \tau^{(m)})}$ for any $x \in
\M{m}{E}$) , then $u$ is completely isometric.
\end{thm}
\begin{rem}
  Note that Theorem \ref{thm=p=2m_m_isometrie_implique_c.isom} is not a
  formal consequence of Theorem
  \ref{thm=p=2m_m_isometrie_implique_c.isom_semifini}. Indeed, when $\un
  \notin E$, assuming \eqref{eq=egalite_des_normes_p_ppair} for any $x \in
  \M m E$ is stronger that assuming that $u: E \rightarrow L_p(\widetilde
  \tau)$ is $m$-isometric, and is weaker than assuming that the map
  $\widetilde u: span(\un,E) \rightarrow L_p(\widetilde \tau)$ that extends
  $u$ by $\widetilde u(\un) = \un$ is $m$-isometric. We therefore give a
  proof of the two results.
\end{rem}

We first provide the proof of Theorem
\ref{thm=p=2m_m_isometrie_implique_c.isom_semifini} which is simpler:
\begin{proof}[Proof of Theorem \ref{thm=p=2m_m_isometrie_implique_c.isom_semifini}]
  Assume that $u$ is $m$-isometric. It clearly suffices to prove that if
  $x_1, \dots x_{2m} \in E$ and $y_i = u(x_i)$, then
\[\tau(x_1^* x_2  x_3^* x_4 \dots x_{2m-1}^* x_{2m}) = \widetilde \tau 
(y_1^* y_2  y_3^* y_4 \dots y_{2m-1}^* y_{2m}).\]

But this is easy to get if one takes $a_1 \dots a_{2m} \in \Mm$
satisfying \eqref{eq=propriete_combinatoire_des_eij} and one applies
$\|x\|_{2m} = \|(\id \otimes u)(x)\|_{2m}$ to $x = x(z_1,\dots z_{2m})
\in \M m E$ defined by
\[ x = \sum_{j=1}^{m} \overline{z_{2j-1}} a_{2j-1}^* \otimes x_{2j-1} + 
  z_{2j} a_{2j} \otimes x_{2j-1}\]
for any $(z_1,\dots z_{2m}) \in \C^{2m}$.

Indeed, $\|x\|_{2m}^{2m}$ is a polynomial in the complex numbers
$z_j$ and $\overline{z_j}$, the coefficient in front of $z_1 z_2 \dots
z_{2m}$ is $\tau(x_1^* x_2 x_3^* x_4 \dots x_{2m-1}^* x_{2m})$.
\end{proof}

\begin{proof}[Proof of Theorem \ref{thm=p=2m_m_isometrie_implique_c.isom}]
  Roughly, the idea of the proof is the same as the previous one: the $2m$
  norm of $\un + \sum a_j \otimes x_j$, depends, as a function of the
  $x_j$'s, only on a finite number of moments of the $x_j$'s. And Lemma
  \ref{thm=egalite_entre_normeP_et_trace_avec_adjoint} shows that these
  moments can be computed from the $2m$-norm of $\un + y$ when $y$
  describes the set of $m \times m$ matrices with values in the linear
  space generated by the $x_j$'s.

But the description of these particular moments is not as simple as in
Theorem \ref{thm=p=2m_m_isometrie_implique_c.isom_semifini}, and the
computations are more complicate.

Take $x \in \M n E$, say $x = \un + \sum_{j=1}^N a_j \otimes x_j$ where
$a_j \in \Mn$ and $x_j \in E$. Denote by $y_j = u(x_j)$. First compute
\begin{eqnarray*}
\|\un + x\|_{2m}^{2m} & = &  \tau^{(n)}( \big(\un + x + x^* + x^* x\big)^m)
%\\& =& \sum_{(X_1, \dots X_m) \in \{1;x;x^*;x^*x\}^m} \tau^{(n)}( X_1 \dotsX_m)
\end{eqnarray*}

The same kind of enumeration as in the proof of Lemma
\ref{thm=lemme_combinatoire} shows that for any integer $j$,
\[\tau^{(n)}\left(\big(x + x^* + x^* x\big)^j \right)= \\
\sum_{k=j}^{2j} \sum_{(\varepsilon_1,\dots,\varepsilon_k) \in \{1,*\}^k} 
\frac j k \binom{\alpha(\varepsilon)}{k-j} \tau^{(m)}\left(x^{\varepsilon_1}
x^{\varepsilon_2}\dots x^{\varepsilon_k}\right).
\]

Multiplying the above equation by $\binom m j$ and summing on $j$ yields to
\begin{multline}
\label{eq=decomp_de_norme_2m_comme_somme}
\tau^{(n)}\left(\big(\un + x + x^* + x^* x\big)^m \right)= \\
\sum_{k=0}^{2m} \sum_{
\begin{array}{c}
\varepsilon_1,\dots,\varepsilon_k \in \{1,*\}\\
i_1,\dots i_k \in \{1,\dots N\}
\end{array}
}
\tr_n\left(a_{i_1}^{\varepsilon_1}\dots a_{i_k}^{\varepsilon_k}\right)
\tau\left(x_{i_1}^{\varepsilon_1} \dots
  x_{i_k}^{\varepsilon_k}\right) \sum_{0\leq j \leq k} \frac j k
\binom{m}{j} \binom{\alpha(\varepsilon)}{k-j}.
\end{multline}

But the assumption \eqref{eq=egalite_des_normes_p_ppair} together with
Lemma \ref{thm=egalite_entre_normeP_et_trace_avec_adjoint} (and Remark
\ref{rem=ppair_on_a_meme_chose}) imply that for any $k \leq 2m$, any
$\varepsilon \in \{1,*\}^k$ and any $(i_1,\dots i_k) \in \{1,\dots N\}^k$,
\[ \tau\left(x_{i_1}^{\varepsilon_1} \dots x_{i_k}^{\varepsilon_k}\right)
\sum_{0\leq j \leq k} \frac j k \binom{m}{j}
\binom{\alpha(\varepsilon)}{k-j} = \widetilde
\tau\left(y_{i_1}^{\varepsilon_1} \dots y_{i_k}^{\varepsilon_k}\right)
\sum_{0\leq j \leq k} \frac j k \binom{m}{j}
\binom{\alpha(\varepsilon)}{k-j}.
\]

Remembering \eqref{eq=decomp_de_norme_2m_comme_somme}, we get that
\[ \|\un + x\|_{L_p(\tau^{(n)})} = \|\un + u^{(n)}(x)\|_{L_p(\widetilde
  \tau^{(n)})}.\]
Since this holds for any $n$ and any $x \in \M n E$, we
have the desired conclusion.
\end{proof}

Now we discuss the case of $p \notin 2\N$. We are unable to determine
whether $n_{p}<\infty$ (or $n_{p,1}<\infty$), but we are able to show that
the assertion $n_{p,1} < \infty$ is related to an assertion concerning the
$*$-distributions of single matricial operators, which we detail below.

If $(x_i)_{i \in I} \in \MM^I$ and $(y_i)_{i \in I} \in \NN^I$ are two
families of operators in von Neumann algebras with \nff{} traces. Then
the same arguments as in Lemma \ref{thm=cas_auto_adjoint} show that
these families have the same $*$-distribution if for any integer $n$,
and any (finitely supported) family $(a_i)_{i \in I} \in \Mn^I$, 
\begin{equation}
\label{eq=egalite_des_etoiles_dist_n}
*-\dist(\sum_{i \in I} a_i \otimes x_i) = *-\dist(\sum_{i\in I} a_i
  \otimes y_i).
\end{equation}

It is also natural to ask: is there an integer $n$ such that
\eqref{eq=egalite_des_etoiles_dist_n} for all $a_i \in \Mn$ imply that
$(x_i)$ and $(y_i)$ have the same $*$-distribution? In the same way as
above, the smallest such integer will be denoted by $N$. (If such integer
does not exist, take $N=\infty$).

Since \eqref{eq=egalite_des_etoiles_dist_n} implies that $\|\un + \sum a_i
\otimes x_i\|_p = \|\sum \un + a_i \otimes y_i\|_p$, Theorem
\ref{thm=equirepartition} shows that when $p$ is not an even integer, $N
\leq n_{p,1}$. To show that $n_{p,1} = \infty$, it would thus be enough to
show $N = \infty$.

\section{Other Applications}
\label{part=resultats_autres}
In this section we prove some other consequences of Lemma
\ref{thm=egalite_entre_normeP_et_trace_avec_adjoint} and Lemma
\ref{thm=egalite_entre_normeP_et_trace_avec_adjoint_nb}. In particular we
prove a \nc{} (weaker) version of Rudin's Theorem
\ref{thm=theoreme_2_rudin}: Theorem \ref{thm=theoreme_2_rudin_nc}. A result
of the same kind (dealing with bounded operators only) and using the same
ideas has already been developed in \cite{MR1978325}. The main difference
is that in \cite{MR1978325}, the author stays at the Banach space level (as
opposed to the operator space level, \ie{}  he does not allow matrix
coefficients in \eqref{eq=hypothese_isometrie_matrices2_2}).

\begin{thm}
\label{thm=theoreme_2_rudin_nc}
Let $(\MM,\tau)$ and $(\NN, \widetilde \tau)$ be as in Theorem
\ref{thm=thm_principal}. Let $0<p<\infty$ and $p \neq 2,4$.  Let $M \subset
L_p(\MM,\tau)$ be a subalgebra (not necessarily self-adjoint) of
$L_p(\MM,\tau)$ containing $\un_\MM$, and let $u:M \rightarrow
L_p(\NN,\widetilde \tau)$ be a linear map such that $u(\un) = \un$.

Assume that $u^{(2)} = \id \otimes u: \M{2}{M} \rightarrow
\M{2}{L_p(\NN,\widetilde \tau)}$ is an isometry for the $p$-``norms'':
\begin{equation}
\label{eq=hypothese_isometrie_matrices2_2}
\forall a \in \M{2}{M} \  \|a\|_p = \|u^{(2)}(a)\|_p.
\end{equation}

Assume moreover that $M \subset L_4(\MM,\tau)$.

Then for all $a,b \in M$
\[u(a b) = u(a) u(b).\]
\end{thm}
\begin{proof}
  The proof is based on Lemma
  \ref{thm=egalite_entre_normeP_et_trace_avec_adjoint_nb}. By Theorem
  \ref{thm=equivalence_entre_dans_L2n_et_image}, $u(M) \subset
  L_4(\NN,\widetilde \tau)$.  

If $a,b \in M$, note that
\begin{multline}
\label{eq=norme_de_u-u2}
\|u(a b) - u(a) u(b)\|_2^2= \widetilde \tau(u(b)^* u(a)^* u(a) u(b))
+ \widetilde \tau(u(a b)^* u(a b))\\ - \widetilde \tau(u(b)^* u(a)^*
u(ab)) -\widetilde \tau(u(a b)^* u(a) u(b)).
\end{multline}

Apply Lemma \ref{thm=egalite_entre_normeP_et_trace_avec_adjoint_nb}
with $n=4$, $(\varepsilon_1,\varepsilon_2,\varepsilon_3,\varepsilon_4)
= (*,*,1,1)$ (so that with the notation of Lemma
\ref{thm=egalite_entre_normeP_et_trace_avec_adjoint}, $\alpha=1$), $m
= n/2 = 2$ and with $(x_1,x_2,x_3,x_4) = ( b, a, a, b)$ on the one
hand and $(x_1,x_2,x_3,x_4) = ( u(b), u(a), u(a), u(b))$ on the other
hand. One gets:
\[\tau(b^* a^* a b) \left( 4 \binom{p/2}{4} +
  3\binom{p/2}{3}\right)= \widetilde \tau(u(b)^* u(a)^* u(a) u(b)) \left( 4
 \binom{p/2}{4} + 3\binom{p/2}{3}\right).\]

But $4 \binom{p/2}{4} + 3\binom{p/2}{3} = p^2
(p/2-1)(p/2-2)/24 \ne 0$ if $p \ne 0, 2, 4$. Thus, 
\[\widetilde \tau(u(b)^* u(a)^* u(a) u(b)) = \tau(b^* a^* a b).\]

The same argument yields to 
\[\widetilde \tau(u(a b)^* u(a) u(b)) = \tau((a b)^* a b) = \tau(b^*
a^* a b),\]
\[\widetilde \tau(u(b)^* u(a)^* u(a b)) = \tau(b^*a^* a b),\]
\[\widetilde \tau(u(a b)^* u(a b)) = \tau(b^*a^* a b).\]

Thus, remembering \eqref{eq=norme_de_u-u2}, one gets
\[\|u(a b) - u(a) u(b)\|_2^2 = 0.\]

Since $\widetilde \tau$ is supposed to be faithful, this implies $u(a
b) = u(a) u(b)$.
\end{proof}

For unital isometries defined on self-adjoint subspaces, the situation
is also nice:
\begin{lemma}
\label{thm=isometrie_preserve_adjoint}
Let $1\leq p<\infty$ and $p \neq 2$.  Let $(\MM,\tau)$ and $(\NN, \widetilde
\tau)$ be as in Theorem \ref{thm=thm_principal}. Let $E \subset L_p(\tau)$
be a unital and self-adjoint subspace (\ie{} $x \in E \Rightarrow x^* \in
E$) and $u : E \rightarrow L_p(\widetilde \tau)$ a unital isometric map.

Then for any $x \in E$ such that $\|x\|_2 < \infty$,
$u(x^*) = u(x)^*$.
\end{lemma}
\begin{proof} 
The proof is of the same kind of the one above: take $x \in E \cap
L_2(\MM)$, and first apply Theorem \ref{thm=isometrie_Lp_est_isom_sur_L2}
to show that $\|u(x)\|_2, \|u(x^*)\|_2 <\infty$. Then the proof consists in
applying Lemma \ref{thm=egalite_entre_normeP_et_trace_avec_adjoint_nb} in
order to prove that $\|u(x^*) - u(x)^*\|_2^2 = 0$. The details are not
provided.
\end{proof}

When the unital completely isometric map $u$ in Theorem
\ref{thm=thm_principal} is defined on the whole $L_p$ space, we
recover some very special cases of known results by Yeadon
\cite[Theorem 2]{MR611284} for isometries and Junge, Ruan and
Sherman \cite[Theorem 2]{MR2255812} for $2$-isometries:
\begin{thm}
\label{thm=cas_ou_E_est_tout_Lp}
Let $p \in \R^+$, $p \neq 2$. Let $u: L_p(\mathcal M,\tau) \rightarrow
L_p(\mathcal N,\widetilde \tau)$ be a linear map such that $u(\un_{\mathcal
  M}) = \un_{\mathcal N}$.

\begin{itemize}
\item If $p \geq 1$ or $u$ maps self-adjoint operators to self-adjoint
  operators, and if $u$ is isometric, then $u$ maps $\MM$ into $\NN$ and
  preserves the trace, the adjoint and the Jordan product: for any $a,b \in
  \mathcal M$
  \[\widetilde \tau(u(a)) = \tau(a) \textrm{, } u(a^*) = u(a)^* \textrm{
    and } u(a b + b a) = u(a) u(b) + u(b) u(a)\]
\item If $u$ is $2$-isometric (\ie{} $u^{(2)}$ is isometric) and $p
  \neq 2$, then the image of $\mathcal M$ is a von Neumann algebra and the
  restriction of $u$ to $\mathcal M$ is a von Neumann algebra trace
  preserving isomorphism.
\end{itemize}
\end{thm}
\begin{proof}
We start by the first point. Take $u$ as above. Note that by Lemma
\ref{thm=isometrie_preserve_adjoint}, if $p \geq 1$ then $u$ preserves
the adjoint.

For any $a \in \MM$ such that $a^* = a$, apply the commutative Theorem
\ref{thm=theoreme_2_rudin} to the unital isometric map $u$ from the
commutative unital subalgebra of $L_p(\MM)$ generated by $a$ into the
commutative unital subalgebra of $L_p(\NN)$ generated by $u(a)$. One
gets that $\|u(a)\|_\infty<\infty$ and that $u(a^2) = u(a)^2$ for any
self-adjoint $a \in M$. By polarization, this implies that for any
self-adjoint $a,b \in M$,
\[ u(a b + ba) = u(a) u(b) + u(b) u(a).\] By linearity this equality
extends to any $a,b \in M$, and $\|a\|_\infty<\infty$. The fact that $u$
preserves the trace is an application of Lemma
\ref{thm=egalite_entre_normeP_et_trace_avec_adjoint_nb} with $n=1$.

Assume now that $u$ is $2$-isometric (with $0< p \neq 2 < \infty$). By
Theorem \ref{thm=equivalence_entre_dans_L2n_et_image}, $u$ (and hence
$u^{(2)}$) preserves the adjoint map. We can apply the isometric case above
for $u^{(2)}$. Thus $u^{(2)}$ is a trace preserving Jordan map. If $a, b
\in \MM$, the equation $u^{(2)}( \widetilde a \widetilde b + \widetilde b
\widetilde a) = u^{(2)}( \widetilde a) u^{(2)}(\widetilde b) +
u^{(2)}(\widetilde b) u^{(2)}(\widetilde a)$ for
\[ \widetilde a = \begin{pmatrix}
0 & a\\0 & 0 \end{pmatrix} \textrm{ and } \widetilde b =
\begin{pmatrix} 0 & 0\\b & 0 \end{pmatrix}\] 

yields to $u(a b) = u(a) u(b)$
(and $u(b a) = u(b) u(a)$). Thus $u$ is a $*$-isomorphism from the von
Neumann algebra $\MM$ onto its image, and it preserves the trace. This
implies, by Lemma \ref{thm=lemme_equidistribution_implique_isomorphisme},
that $u(\MM)$ is a von Neumann subalgebra of $\NN$ and that $u$ is a von
Neumann isomorphism.
\end{proof}

\subsection{Applications to \nc{} $H^p$ spaces}
The main result also applies in the setting of \nc{} $H^p$ spaces
(see \cite{Blecher_Labu}).

\begin{dfn} 
  Let $(\MM, \tau)$ be, as above, a von Neumann algebra with a faithful
  normal normalized trace.  A tracial subalgebra $A$ of $(\mathcal M,
  \tau)$ is a weak-$\ast$ closed unital subalgebra such that the
  conditional expectation $\phi_A:\mathcal M \rightarrow A \cap A^*$
  satisfies $\phi_A(a b) = \phi_A(a) \phi_A(b)$ for any $a, b \in A$.

  A \nc{} $H^p$ space is the closure, denoted $[A]_p$, of a
  tracial subalgebra $A$ in $L_p(\mathcal M,\tau)$.
\end{dfn}

Since by definition, a \nc{} $H^p$ space is a unital subspace of
$L_p(\MM,\tau)$ in which the subset of bounded operators is dense, Theorem
\ref{thm=thm_principal} automatically implies the following, which gives a
beginning of answer to a question raised in \cite{Blecher_Labu} and
\cite{MR1978325}:

\begin{thm}
Let $p \notin 2\N$.

A unital complete isometry between \nc{} $H^p$ spaces
extends to an isomorphism between the von Neumann algebras they
generate.

Moreover, if a non commutative $H^p$ space is unitally completely isometric
to a unital subspace $E$ of a \nc{} $L_p$ space, then $E$ is a
\nc{} $H^p$ space.
\end{thm}

For a $2$-isometric map between \nc{} $H_p$-spaces, we also get
something:
\begin{thm}
Let $p \in \R$, $p\neq 2, 4$.

Let $u$ be a $2$-isometric and unital map from a \nc{} $H^p$ space
$[A]_p$ into a \nc{} $L_p$ space $L_p(\NN, \widetilde \tau)$.

Then the image of $A$ is a subalgebra $B$ of $\NN$ such that for any
$a,b \in A$,
\[ u(a b) = u(a) u(b). \]

Moreover, $B \cap B^* = u(A \cap A^*)$ is a von Neumann algebra such
that the restriction to $B$ of the conditional expectation $\Phi_B :
\NN \rightarrow B \cap B^*$ satisfies, for any $a \in M$
\[ \Phi_B(u(a)) = u(\Phi_A(a)).\]
\end{thm}
\begin{proof}
The fact that $B \egdef u(A)$ is contained in $\NN$ follows from
Theorem \ref{thm=equivalence_entre_dans_L2n_et_image}.

Theorem \ref{thm=theoreme_2_rudin_nc} then implies that $B$ is an
algebra and that $u(a b) = u(a) u(b)$.

The fact that $B \cap B^* = u(A \cap A^*)$ is immediate from Lemma
\ref{thm=isometrie_preserve_adjoint}, and Theorem
\ref{thm=cas_ou_E_est_tout_Lp} shows that $u(A \cap A^*)$ is a von
Neumann algebra. 

Recall that the conditional expectation $\Phi_B : \NN \rightarrow B
\cap B^*$ coincides with the orthogonal projection for $L_2(\NN)$ to
$L_2(B \cap B^*)$. To check the last equation
\[ \Phi_B(u(a)) = u(\Phi_A(a))\]
we thus have to show that for any $x \in B \cap B^*$,
\[ \widetilde \tau\left(x u(a)\right) = \widetilde \tau\left(x u(\Phi_A(a))\right).\]

Write $x = u(b)$ for $b \in \MM$. The above equation arises from the
definition of $\Phi_A(a)$, from the multiplicativity of $u$ and from
the fact that, by Lemma
\ref{thm=egalite_entre_normeP_et_trace_avec_adjoint_nb}, $\widetilde
\tau \circ u = \tau$:
\begin{eqnarray*}
\widetilde \tau\left(x u(a)\right) &=& \widetilde \tau\left(u(b
a)\right) \\
& = & \tau\left(b a\right)\\
& = & \tau\left(b \Phi_A(a)\right)\\
& = & \widetilde \tau\left(u(b^* \Phi_A(a))\right)\\
& = & \widetilde \tau\left(x^* u(\Phi_A(a))\right).
\end{eqnarray*}
This concludes the proof.
\end{proof}

\bigskip

We end this paper with some additional remarks and questions related
Yeadon's result. The main theorem of \cite{MR611284} in
particular contains the following:
\begin{lemma}
\label{thm=image_de_un_commute_avec_le_reste}
Let $u$ be an isometry from an $L_p$-space $L_p(\MM)$ constructed on a von
Neumann algebra $(\MM,\tau)$ with a \nff{} trace to an $L_p$-space
$L_p(\NN)$ constructed on a von Neumann algebra $(\NN,\sigma)$ with a
normal \emph{semifinite} faithful trace. 

Then if $u(\un) = b$ is positive, then $b$ commutes with $u(L_p(\MM))$
and has full support.
\end{lemma}
\begin{rem}
This also holds if $\NN$ does not carry a semifinite
trace and $L_p(\NN)$ is Haagerup's generalized $L_p$ space (see
\cite[Theorem 3.1]{MR2255812}).
\end{rem}

This fact allows to reduce the general case to the unital case. Here
are the details: if one denotes by $b = u(\un) \geq 0 \in L_p(\mathcal
M)$, by $s \in \NN$ the support projection of $b$, and by $\widetilde
\NN$ the von Neumann subalgebra of $s \NN s$ generated by $b^{-1}
u(\MM)$, then $\widetilde \NN$ carries a \nff{} trace given by
\[\widetilde \tau(x) =  \left \{ \begin{array}{ll}
\sigma( b^p x) & \textrm{if $\sigma$ was a semifinite trace on $\NN$}\\
tr(b^p x). & \textrm{in Haagerup's construction.}
\end{array}\right.\]
(in Haagerup's construction, $tr$ is the trace functional on $L_1(\NN)$).

The assumption that $u$ is an isometry exactly means that (the
unital linear map) $b^{-1} \cdot u$ is an isometry from $L_p(\MM,\tau)$
to $L_p(\widetilde \NN,\widetilde \tau)$. 

Thus Yeadon's Lemma \ref{thm=image_de_un_commute_avec_le_reste} (resp. with
the preceding remark) and Theorem \ref{thm=cas_ou_E_est_tout_Lp} of this
paper are enough to recover Yeadon's result (resp. Junge, Ruan and
Sherman's result with the restriction that the first $L_p$-space be
semifinite). Of course, all this is not surprising at all since Lemma
\ref{thm=image_de_un_commute_avec_le_reste} contains most of the results
from \cite{MR611284}, and we therefore do not provide more details. But
this leads naturally to the question: to what extend Lemma
\ref{thm=image_de_un_commute_avec_le_reste} can be generalized when $u$ is
only defined on a unital subspace of $L_p(\MM,\tau)$?

As justified above, it is natural to wonder whether the same result
holds for isometries between subspaces of \nc{} $L_p$
spaces. More precisely, let $\un \in E \in L_p(\MM,\tau)$ be a unital
subspace of a \nc{} $L_p$ space with $\tau$ a \nff{}
trace. Let $u: E \rightarrow L_p(\NN)$ be an isometry between $E$ and
a subspace of an arbitrary \nc{} $L_p$ space such that
$u(\un) \geq 0$. Then is it true that $u(\un)$ commutes with $u(E)$
and has full support in $u(E)$? (that is: if $s$ is the support
projection of $u(\un)$, then $s u(x) = u(x) s = u(x)$ for any $x \in
E$). As noted above, this would allow to use all the results of this
paper for $u$ and would have several interesting consequences.

It should be noted that Yeadon's proof (as well as the generalization
in \cite{MR2255812}) consists in applying the equality condition in
Clarkson's inequality for projections with disjoint supports. This of
course is not possible for a general subspace of $L_p(\mathcal M)$
since it may not contain any nontrivial projection.

\bibliographystyle{plain} \bibliography{biblio}

\def\cprime{$'$}
\begin{thebibliography}{10}

\bibitem{MR2284176}
Rajendra Bhatia.
\newblock {\em Positive definite matrices}.
\newblock Princeton Series in Applied Mathematics. Princeton University Press,
  Princeton, NJ, 2007.

\bibitem{Blecher_Labu}
David~P. Blecher and Louis~E. Labuschagne.
\newblock Von {N}eumann algebraic {$H^p$} theory.
\newblock In {\em Proceedings of the fifth conference on function spaces},
  Contemporary Mathematics, pages 89--114. American Mathematical Society, 2007.

\bibitem{A_CITER}
Beno\^it Collins and Ken Dykema.
\newblock A linearization of {C}onnes' embedding problem.
\newblock {\em arXiv:0706.3918v1}, 2007.

\bibitem{MR840845}
Thierry Fack and Hideki Kosaki.
\newblock Generalized {$s$}-numbers of {$\tau$}-measurable operators.
\newblock {\em Pacific J. Math.}, 123(2):269--300, 1986.

\bibitem{MR1957004}
Richard~J. Fleming and James~E. Jamison.
\newblock {\em Isometries on {B}anach spaces: function spaces}, volume 129 of
  {\em Chapman \& Hall/CRC Monographs and Surveys in Pure and Applied
  Mathematics}.
\newblock Chapman \& Hall/CRC, Boca Raton, FL, 2003.

\bibitem{MR0169081}
Frank Forelli.
\newblock The isometries of {$H\sp{p}$}.
\newblock {\em Canad. J. Math.}, 16:721--728, 1964.

\bibitem{MR0318897}
Frank Forelli.
\newblock A theorem on isometries and the application of it to the isometries
  of {$H\sp{p}(S)$} for {$2<p<\infty $}.
\newblock {\em Canad. J. Math.}, 25:284--289, 1973.

\bibitem{MR611233}
Clyde~D. Hardin, Jr.
\newblock Isometries on subspaces of {$L\sp{p}$}.
\newblock {\em Indiana Univ. Math. J.}, 30(3):449--465, 1981.

\bibitem{MR2255812}
Marius Junge, Zhong-Jin Ruan, and David Sherman.
\newblock A classification for 2-isometries of noncommutative {$L\sb
  p$}-spaces.
\newblock {\em Israel J. Math.}, 150:285--314, 2005.

\bibitem{MR1863709}
Alexander Koldobsky and Hermann K{\"o}nig.
\newblock Aspects of the isometric theory of {B}anach spaces.
\newblock In {\em Handbook of the geometry of Banach spaces, Vol. I}, pages
  899--939. North-Holland, Amsterdam, 2001.

\bibitem{MR1978325}
L.~E. Labuschagne.
\newblock Analogues of composition operators on non-commutative {$H\sp
  p$}-spaces.
\newblock {\em J. Operator Theory}, 49(1):115--141, 2003.

\bibitem{MR1648908}
Gilles Pisier.
\newblock Non-commutative vector valued {$L\sb p$}-spaces and completely
  {$p$}-summing maps.
\newblock {\em Ast\'erisque}, (247):vi+131, 1998.

\bibitem{MR2006539}
Gilles Pisier.
\newblock {\em Introduction to operator space theory}, volume 294 of {\em
  London Mathematical Society Lecture Note Series}.
\newblock Cambridge University Press, Cambridge, 2003.

\bibitem{MR1999201}
Gilles Pisier and Quanhua Xu.
\newblock Non-commutative {$L\sp p$}-spaces.
\newblock In {\em Handbook of the geometry of Banach spaces, Vol.\ 2}, pages
  1459--1517. North-Holland, Amsterdam, 2003.

\bibitem{MR0621703}
A.~I. Plotkin.
\newblock An algebra that is generated by translation operators, and
  {$L\sp{p}$}-norms.
\newblock In {\em Functional analysis, No. 6: Theory of operators in linear
  spaces (Russian)}, pages 112--121. Ul\cprime janovsk. Gos. Ped. Inst.,
  Ul\cprime yanovsk, 1976.

\bibitem{MR0410355}
Walter Rudin.
\newblock {$L\sp{p}$}-isometries and equimeasurability.
\newblock {\em Indiana Univ. Math. J.}, 25(3):215--228, 1976.

\bibitem{MR2215435}
David Sherman.
\newblock On the structure of isometries between noncommutative {$L\sp p$}
  spaces.
\newblock {\em Publ. Res. Inst. Math. Sci.}, 42(1):45--82, 2006.

\bibitem{MR611284}
F.~J. Yeadon.
\newblock Isometries of noncommutative {$L\sp{p}$}-spaces.
\newblock {\em Math. Proc. Cambridge Philos. Soc.}, 90(1):41--50, 1981.

\end{thebibliography}
\end{document}